\documentclass[1pt]{article}
\usepackage{amsmath, amssymb, amsthm, amsfonts, cases}
\usepackage{mathrsfs}
\usepackage{url}
\usepackage{authblk}
\usepackage[usenames]{color}
\usepackage{geometry}
\allowdisplaybreaks

\theoremstyle{plain}
\newtheorem{thm}{Theorem}[section]

\newtheorem{pro}[thm]{Problem}
\newtheorem{lem}[thm]{Lemma}

\theoremstyle{definition}
\newtheorem{defn}{Definition}[section]
\newtheorem{ass}{Assumption}[section]
\newtheorem{rmk}{Remark}[section]

\newcommand{\eps}{\varepsilon}

\newcommand{\su}{\sum_{i=1}^\infty}
\newcommand{\la}{\langle}
\newcommand{\ra}{\rangle}

\makeatletter\@addtoreset{equation}{section} \makeatother
\begin{document}

\title{  Stochastic Evolution Equation Driven by Teugels Martingale and Its Optimal Control
 \thanks{This work was supported by the Natural Science Foundation of Zhejiang Province
for Distinguished Young Scholar  (No.LR15A010001),  and the National Natural
Science Foundation of China (No.11471079, 11301177) }}

\date{}

   \author{ Qingxin Meng\thanks{Corresponding author.   E-mail: mqx@zjhu.edu.cn}  \hspace{1cm}
   Qiuhong Shi
 \hspace{1cm}  Maoning Tang 
\hspace{1cm}
\\\small{Department of Mathematics, Huzhou University, Zhejiang 313000, China}}

\maketitle
\begin{abstract}

The paper is concerned with a class of stochastic evolution
equations in Hilbert space with random coefficients driven by Teugel's martingales and
an independent multi-dimensional Brownian motion and its optimal control problem. Here Teugels
martingales are a family of pairwise strongly orthonormal
martingales associated with L\'{e}vy processes (see Nualart and
Schoutens \cite{NuSc}). 
 There are
three major ingredients. The first is to
prove
the existence and uniqueness of the solutions by continuous dependence theorem of solutions combining with  the  parameter extension method. The second is
to establish the stochastic maximum principle
and verification theorem for our optimal control problem by the classic
convex variation method and dual technique.
The third is to represent an example of   a Cauchy problem for a controlled stochastic partial differential equation driven by  Teugels martingales which our theoretical results can solve.

\end{abstract}


\maketitle

\section{ Introduction }

This  paper is concerned with  the  following
stochastic evolution equation in Hilbert
space  \begin{eqnarray} \label{eq:1.1}
  \left\{
  \begin{aligned}
   d X (t) = & \ [ A (t) X (t) + b ( t, X (t)) ] d t
+ [B(t)X(t)+g( t, X (t)) ]d W(t)
 + \su\sigma^i(t, X(t-)) dH^i(t),  \\
X (0) = & \  x , \quad t \in [ 0, T ]
  \end{aligned}
  \right.
\end{eqnarray}
in the framework of a Gelfand triple $V \subset H= H^*\subset V^{*},$  where $ H$ and  $V$ are
two given Hilbert spaces. Here on  a given  filtrated probability space  $(\Omega, \mathscr{F},\{
{\mathscr F}_t\}_{0\leq t\leq T}, P),$ $W$  is  a  one-dimensional  Brownian motion  and  $\{H^i(t), 0\leq t \leq T\}_{i=1}^\infty$ is
is Teugels martingale associated with 
a one-dimensional L\'{e}vy process $\{L(t),
0\leq t \leq T\}$,
$A:[0,T]\times \Omega \longrightarrow {\mathscr L} (V, V^*)$, $B
  : [0,T]\times \Omega \longrightarrow {\mathscr L} (V, H ),$ $b:[0,T]\times  \Omega\times H  \longrightarrow H$, $g:[0,T]\times\Omega\times H  \longrightarrow H$ and $\sigma^i:[0,T]\times  \Omega \times
   E\times H  \longrightarrow H$  are given random mappings. Here we denote by  $\mathscr{L}(V,V^*)$ the space of bounded
linear transformations of V into $V^*$, by ${\mathscr L} (V, H)$   the space of bounded
linear transformations of  $H$ into $V.$
    An adapted solution of
   \eqref{eq:1.1} is  a $V$-valued, $\{{\mathscr F}_t\}_{0\leq t\leq T}$-adapted process $X(\cdot)$  which satisfies \eqref{eq:1.1} under some appropriate sense.
    Such a model as \eqref{eq:1.1}
    represents a large  classes of stochastic partial differential equations, for instance  the nonlinear filtering equation and other stochastic parabolic PDEs (cf.\cite{Chow07}), but it is
    by no means the largest one. Partial
    differential equation are too diverse to be covered by a single model, like ordinary equations.

In 2000, Nualart and Schoutens \cite{NuSc}
got a martingale representation theorem for a type of L\'{e}vy
processes through Teugels martingales which
are a family of pairwise strongly orthonormal martingales associated
with L\'{e}vy processes. Later, they proved in \cite{NuSc01} the
existence and uniqueness theory of BSDE driven by Teugels
martingales. The above results are further extended to the
one-dimensional BSDE driven by Teugels martingales and an
independent multi-dimensional Brownian motion by Bahlali et al
\cite{BEE03}. One can refer to \cite{Otm06, Otm08, ReFa09, ReOt10}
for more results on such kind of BSDEs. 

In the mean time, the stochastic optimal control problems related to
Teugels martingales were studied. In 2008, a stochastic
linear-quadratic problem with L\'{e}vy processes was considered by
Mitsui and Tabata \cite{MiTa08}, in which they established the
closeness property of a multi-dimensional backward stochastic Riccati
differential equation (BSRDE) with Teugels martingales and proved
the existence and uniqueness of the solution to such kind of
one-dimensional BSRDE, moreover, in their paper an application of
BSDE to a financial problem with full and partial observations was
demonstrated. Motivated by \cite{MiTa08}, Meng and Tang
\cite{MeTa08} studied the general stochastic optimal control problem
for the forward stochastic systems driven by Teugels martingales
and an independent multi-dimensional Brownian motion, of which the
necessary and sufficient optimality conditions in the form of
stochastic maximum principle with the convex control domain are
obtained. In 2012, Tang and Zhang \cite{TaZh10} studied the optimal
control problem  of backward stochastic systems driven by Teugels
martingales and an independent multi-dimensional Brownian motion and
obtained the corresponding stochastic maximum principle.

Due to the
interesting analytical contents  and
wide applications in various
sciences such as physics,
mechanical engineering, control
theory and economics,  the theory of
SPDEs  driven by Wiener processes or
 Gaussian random processes now has
been investigated extensively  and
has already achieved fruitful results on
the existence  uniqueness, stability,
invariant measure   other quantitative and
qualitative properties of solution and so on.
 There are a great amount of literature on this topic, for example  \cite{Chow07}\cite{Da}\cite{Ro} and references therein. On the one hand, non-Gaussian random processes play an increasing role in modeling stochastic dynamical systems. Typical examples of non-Gaussian stochastic processes are L\'{e}vy processes and processes arising by Poisson random measures. In neurophysiology the driving noise of the cable equation is basically impulsiv  e.g., of a Poisson type (see \cite{Wal} ) or, on the other hand, Woyczy\'{n}ski describes in
 \cite{woy} a number of phenomena from fluid mechanics, solid state physics, polymer chemistry, economic science  etc., for which non-Gaussian L\'{e}vy processes can be used as their mathematical model in describing the related stochastic behavior. Thus, from the point of view of applications one might feel that the restriction to Wiener processes or Gaussian noise is unsatisfactory; to handle such cases one can replace Wiener processes or Gaussian noise by a Poisson random measure.
Most recently, thanks to comprehensive practical applications,  many attentions have been paid to SPDEs driven by jump processes, (cf., for exampl \cite{Al}\cite{RSS}\cite{Ren}\cite{Ro}
\cite{Sa}\cite{Ya}\cite{Zhao}\cite{Zhai}
 and the references therein).    It is worth mentioning that R\"{o}cker and Zhang
\cite{Ro}  established the existence and uniqueness theory for solutions of stochastic evolution equations of type (1.1) by a successive approximations, in which case  the operator $B$ does not exist.

 One of the purposes of this paper  is to establish the continuous dependence of  the solution on the
  coefficients and  the existence and uniqueness of solutions    to  the stochastic evolution equation \eqref{eq:1.1}.
It is well known that there are
  two
different methods to analyzing SPDEs:   the semigroup (or mild solution ) approach (cf. \cite{Da}) and the variational approach (cf.\cite{Pr}). For \eqref{eq:1.1}, since its
 coefficients are allowed to be random,
  we need to use
  the variational approach in the weak solution framework (in the PDE sense) of the Gelfand triple and can not use the mild solution  approach to study it.
 In fact,  when the coefficients
 are deterministic, we always
 study the stochastic
 evolution equation in
 the mild
 solution framework.
 However, due to the randomness of
coefficients, it seems very difficult or even impossible to tackle
the problem  in the mild
solution sense.   Indeed,  if we define
the mild solution as usual, the adaptability of the integrand in the
stochastic integral may not be satisfied due to the randomness of
the operator $A$. The advantage of the variational approach is that a version of
It\^o's formula exists in the context of the Gelfand triple of
Hilbert spaces (see \cite{Gy} for details).
Such a formula will play an important role in proving the main
results throughout this paper.

 Another  purpose of this paper is to  establish the
 maximum principle and verification theorem
 for the optimal control problem where the state process is driven by a  controlled stochastic evolution equation \eqref{eq:1.1}.
 A classical approach for optimal control problems is to derive necessary conditions satisfied by an optimal control, such as Pongtryagin¡¯s maximum principle. Since the 1970s, the maximum principle has been extensively studied for stochastic control systems: in the finite-dimensional case  it has been solved by Peng \cite{Peng} in a general setting where the control was allowed to take values in a nonconvex set and enter into the diffusion, while in the infinte-dimensional case  the existing literatur  e.g.,\cite{Ben830},
 \cite{Hu} \cite{Hupeng}\cite{Zhou}, required at least one of the following three assumptions: (1) the control domain was convex, (2) the diffusion did not depend on the control, (3) the state equation and cost functional were both linear in the state variable. So far the general maximum principle for infinite-dimensional stochastic control systems, the counterpart of Peng's result, remained open for a long time.
  Until recently, Du and Meng attempt to fill  this gap in \cite{Du} where they
   developed a new procedure to perform the second-order duality analysis: by virtue of the Lebesgue differentiation theorem and an approximation argument to establish the corresponding maximum principle. Meanwhile  other  very important works concerned with the general stochastic maximum principle in infinite dimensions   were given in  \cite{Fu} \cite{Lu} besides \cite{Du}.  From the
   above references,  works on optimal control problems of infinite dimension stochastic evolution equation or stochastic partial differential equation are mainly  concerned with systems driven only by Wiener  processes. In contrast, there have not been
  a
 number of
  results on  the optimal control for stochastic partial
differential equations driven by jump processes. In 2005,  {\O}ksendal, Prosk and    Zhang \cite{Ok} studied the optimal control problem  of quasilinear semielliptic SPDEs driven by Poisson random measure   and gave sufficient maximum principle results, not necessary ones. In 2017,  Tang and Meng
 \cite{Tangmeng2016}studied the optimal control problem of more general stochastic
 evolution equations driven by Poisson random measure with random coefficients   and gave necessary and  sufficient maximum principle results.In this paper,
for a controlled  stochastic evolution equation \eqref{eq:1.1},  we suppose
 that the  control domain is convex.  We adopt the  convex variation method and the first adjoint duality analysis to show a necessary maximum principle
 where  the continuous dependence
 theorem (see Theorem \ref{thm:3.2}) plays a key role in proving the variation inequality for the cost functional
 (see Lemma \ref{lem:6.2}).  Under the convexity assumption of the Hamiltonian and the terminal cost, we provide  a sufficient maximum principle for this optimal problem which is the so-called verification theorem.
   It is worth mentioning that if the admissible control set is non-convex and  the diffusion terms of  the state equation is independent of the control variable  we can
use the first-order spike variation method to obtain the maximum principle in the global
form  by establishing  some subtle   $L^2$-estimate for the state equation. All the details shall be given in our forthcoming paper.  But for the general setting,  it seems very difficult or even impossible to obtain the corresponding
maximum principle because it seems impossible to establish  some $L^p$ $(p>2)$ estimates as in \cite{Du} for the state
  process which play a key role in
  the second variation analysis.
  Finally, to illustrate our
results, we apply the stochastic maximum principles to solve an optimal control of a Cauchy problem for
a controlled stochastic linear
partial differential equation.

The rest of this
paper is structured as follows. In Section 2,
 we provide the basic notations and recall It\^{o} formula for Teugels martingales in Hilbert space used  frequently in this paper.
 Section 3  establishes the continuous dependence and  the existence and
 uniqueness of solutions to the
 stochastic evolution equation \eqref{eq:1.1}.
 Section 4 formulates the optimal control problem specifying the hypotheses. In section 5, adjoint equation is introduced which turns out to be a backward stochastic
evolution equation driven by Teugels martingales. In Section 6, we establish the  stochastic  maximum principle by the classical convex variation method. In Sections 7, the verification theorem for optimal controls  is obtained by dual technique. In section 8, we present an application of our results.  The final section concludes the paper.

\section{Notations and It\^{o} Formula for Teugels Martingales in Hilbert Space  }

Let $(\Omega, \mathscr{F},\{\mathscr{F}_t\}_{0\leq t\leq T}, P)$ be
a complete probability space. The filtration
$\{\mathscr{F}_t\}_{0\leq t\leq T}$ is right-continuous and
generated by a one-dimensional standard Brownian motion $\{W(t),
0\leq t\leq T\}$ and a one-dimensional L\'{e}vy process $\{L(t),
0\leq t \leq T\}$. It is known that $L(t)$ has a characteristic
function of the form $$\mathbb E\big[e^{i\theta L(t)}\big]=\exp\bigg[ia\theta
t-{1\over2}\sigma^2\theta^2t+t\int_{\mathbb{R}^1}(e^{i\theta
x}-1-i\theta x I_{\{|x|<1\}})v(dx)\bigg],\quad \forall \theta\in \mathbb{R},$$ where
$a\in\mathbb{R}^1$, $\sigma>0$ and $v$ is a measure on
$\mathbb{R}^1$ satisfying (i)$\displaystyle \int_{-T}^T(1\wedge
x^2)v(dx)<\infty$ and (ii) there exists $\varepsilon >0$ and
$\lambda >0$, s.t. $\displaystyle \int_{\{-\varepsilon,
\varepsilon\}^c} e^{\lambda |x|}v(dx)<\infty$. These settings imply
that the random variables $L(t)$ have moments
of all orders. 
Denote by $\mathscr{P}$  the
predictable sub-$\sigma$ field of $\mathscr B([0, T])\times
\mathscr{F}$, then we introduce the following notation used
throughout this paper.

$\bullet$~~$X$: a Hilbert space with norm $\|\cdot\|_X$.

$\bullet$~~$(\cdot,\cdot)_X:$ the inner product in Hilbert space $X.$

$\bullet$~~$l^2$: the space of all real-valued sequences
$x=(x_n)_{n\geq 1}$ satisfying
$$\|x\|_{l^2}\triangleq\sqrt{\displaystyle \sum_{i=1}^\infty x_i^2}<+\infty.$$

$\bullet$~~$l^2(X):$ the space of all H-valued sequence
$f=\{f^i\}_{i\geq 1}$ satisfying
$$\|f\|_{l^2(X)}
\triangleq\sqrt{\displaystyle
\sum_{i=1}^\infty||f^i||_X^2}<+\infty.$$

$\bullet$~~$l_{\mathscr{F}}^2(0, T; X):$ the space of all
$l^2(X)$-valued and ${\mathscr{F}}_t$-predictable processes
$f=\{f^i(t,\omega),\ (t,\omega)\in[0,T]\times\Omega\}_{i\geq1}$
satisfying
$$\|f\|_{l_{\mathscr{F}}^2(0, T, X)}\triangleq\sqrt{E\displaystyle
\int_0^T\sum_{i=1}^\infty||f^i(t)||_X^2dt}
<\infty.$$

$\bullet$~~$M_{\mathscr{F}}^2(0,T;X):$ the space of all $X$-valued
and ${\mathscr{F}}_t$-adapted processes $f=\{f(t,\omega),\
(t,\omega)\in[0,T]\times\Omega\}$ satisfying
$$\|f\|_{M_{\mathscr{F}}^2(0,T;X)}
\triangleq\sqrt{E\displaystyle
\int_0^T\|f(t)\|_X^2dt}<\infty.$$

$\bullet$~~$S_{\mathscr{F}}^2(0,T;X):$ the space of all $X$-valued
and ${\mathscr{F}}_t$-adapted  c\`{a}dl\`{a}g processes
$f=\{f(t,\omega),\ (t,\omega)\in[0,T]\times\Omega\}$ satisfying
$$\|f\|_{S_{\mathscr{F}}^2(0,T;X)}\triangleq\sqrt{E\displaystyle\sup_{0 \leq t \leq T}\|f(t)\|_X^2dt}<+\infty.$$

$\bullet$~~$L^2(\Omega,{\mathscr{F}},P;X):$ the space of all
$H$-valued random variables $\xi$ on $(\Omega,{\mathscr{F}},P)$
satisfying
$$\|\xi\|_{L^2(\Omega,{\mathscr{F}},P;X)}\triangleq E\|\xi\|_X^2<\infty.$$

We denote by $\{H^i(t), 0\leq t \leq T\}_{i=1}^\infty$ the Teugels
martingales
associated with the L\'{e}vy process $\{L(t),0\leq t \leq T\}$. 
$H^i(t)$ is given by
$$
H^i(t)=c_{i,i}Y^{(i)}(t)+c_{i,i-1}Y^{(i-1)}(t)+\cdots+c_{i,1}Y^{(1)}(t),
$$
where $Y^{(i)}(t)=L^{(i)}(t)-E[L^{(i)}(t)]$ for all $i\geq 1$,
$L^{(i)}(t)$ are so called power-jump processes with
$L^{(1)}(t)=L(t)$, $L^{(i)}(t)=\displaystyle\sum_{0<s\leq t}(\Delta
L(s))^i$ for $i\geq 2$ and the coefficients $c_{ij}$ corresponding to
the orthonormalization of polynomials $1,x, x^2,\cdots$ w.r.t. the
measure $\mu(dx)=x^2v(dx)+\sigma^2\delta_0(dx)$. The Teugels
martingales $\{H^i(t)\}_{i=1}^\infty$ are pathwise strongly
orthogonal 
and their predictable quadratic variation processes are given by
$$\langle H^{(i)}(t), H^{(j)}(t)\rangle=\delta_{ij}t$$
For more details of Teugels martingales, we invite the
reader to consult Nualart and Schoutens
\cite{NuSc}.

Let $V$ and $H$ be two separable (real) Hilbert spaces such that $V$
is densely embedded in $H$. We identify $H$ with its dual space by
the Riesz mapping. Then  we can take $H$ as a pivot space and  get a
Gelfand triple $V \subset H= H^*\subset V^{*},$ where  $H^*$ and
$V^{*}$ denote the dual spaces of $H$ and $V$, respectively. Denote
by $\|\cdot\|_{V},\|\cdot\|_{H}$ and $\|\cdot\|_{V^*}$ the norms of
$V,H$ and $V^*$, respectively, by $(\cdot,\cdot)_H$ the inner
product in $H$, by $\la\cdot,\cdot\ra$ the duality product between
$V$ and $V^{*}$.
 Moreover we write $\mathscr{L}(V,V^*)$ the space of bounded
linear transformations of V into $V^*$.
Throughout this paper, we let  $C$ and $K$  be two generic positive constants, which may be different from line to line.

Now we present an It\^o's formula in Hilbert space
which will be  frequently  used  in this paper.

\begin{lem}\label{lem:c1}
Let $\varphi\in L^{2}(\Omega,\mathscr{F}_{0},P; H)$. Let $Y, Z, \Gamma$ and $R\equiv (R^i)_{i=1}^{\infty}$  be three
  progressively measurable stochastic processes defined on $[0,T]\times \Omega$ with values in $V,H$ and $V^{*}$ such that $ Y\in { M}_{\mathscr F}^2 (0, T; V),  Z\in {M}_{\mathscr F}^2 (0, T; H),$ $\Gamma \in { M}_{\mathscr F}^2 (0, T; V^*) $ and $R\in l_{\mathscr{F}}^2(0, T; H)$, respectively. 
 Suppose that
for every
  $\eta \in V$ and almost every $(\omega,t)\in\Omega\times[0,T]$, it holds
  that
  \begin{equation*}
    ( \eta,Y(t))_H =( \eta, \varphi)_H+
    \int_{0}^{t} \la \eta,\Gamma(s) \ra ds
  +\int_0^t( \eta, Z(s) )_HdW(s)+\sum_{i=1}^\infty\int_0^t
     ( \eta, R^i(s) )_HdH^i(s).
  \end{equation*}
  Then  $Y$ is a $H$-valued  strongly c\`{a}dl\`{a}g
  $\mathscr F_t$-adapted process 
such that  the following It\^{o} formula holds

 \begin{eqnarray}\label{eq:2.1}
   \begin{split}
     ||Y(t)||_H^2=&||\varphi||^2+
     2\int_0^t\langle \Gamma(s), Y(s) \rangle
     ds +2\int_0^t ( Y(s), Z(s))_H
     dW(s)+ \int_0^t||Z(s)||_H^2ds
     \\&+2\sum_{i=1}^\infty\int_0^t 
     (Y(s), R^i(s))_HdH^i(s)
     +\sum_{i=1}^\infty
     \sum_{j=1}^\infty\int_0^t (R^i(s), R^j(s))_Hd[H^i(s), H^j(s)].
   \end{split}
 \end{eqnarray}
\end{lem}

\begin{proof}
  The proof follows that of Theorem 1 in Gy\"{o}ngy and Krylov [13].
\end{proof}

\section{Stochastic Evolution Equation 
driven by Teugels Martingales}

In this section, we present some
 preliminary results of the following  stochastic evolution equation (SEE  for short) driven by Brownian motion $\{W(t),
0\leq t\leq T\}$  and  Teugels Martingales $\{H^i(t), 0\leq t \leq T\}_{i=1}^\infty$:

\begin{eqnarray} \label{eq:3.1}
  \left\{
  \begin{aligned}
   d X (t) = & \ [ A (t) X (t) + b ( t, X (t)) ] d t
+ [B(t)X(t)+g( t, { X (t)}) ]d W(t)
 + \sum_{i=1}^\infty\sigma^i(s,X(s-))dH^i(s),  \\
X (0) = & \  x \in H , \quad t \in [ 0, T ],
  \end{aligned}
  \right.
\end{eqnarray}
where $A,B,b,g$ and  $\sigma\equiv{(\sigma^i)}_{i=1}^\infty$ are given random mappings
which satisfy the following
standard  assumptions.

\begin{ass} \label{ass:3.1}
  The operator processes $A:[0,T]\times \Omega \longrightarrow {\mathscr L} (V, V^*)$ and $B
  : [0,T]\times \Omega \longrightarrow {\mathscr L} (V, H)$
  are weakly predictable; i.e.,
  $ \langle A(\cdot)x, y \rangle$ and $(B(\cdot)x, y)_H$
  are both predictable process for every $x, y\in V, $
  and satisfy the coercive condition, i.e.,  there exist
  some constants  $ C, \alpha>0$ and $\lambda$ such that for any $x\in V$ and  each $(t,\omega)\in [0,T]\times \Omega,$
    \begin{eqnarray}
    \begin{split}
     - \langle A(t)x, x \rangle +\lambda ||x||_H^2\geq \alpha
      ||x||_V^2+||Bx||_H^2,
    \end{split}
  \end{eqnarray}
  and  \begin{eqnarray} \label{eq:3.3}
\sup_{( t, \omega ) \in [0, T] \times \Omega} \| A ( t,\omega ) \|_{{\mathscr L} ( V, V^* )}
 +\sup_{( t, \omega ) \in [0, T] \times \Omega} \| B ( t,\omega ) \|_{{\mathscr L} ( V, H )} \leq C \ .
\end{eqnarray}
\end{ass}

\begin{ass} \label{ass:3.2}
   The mappings $b:[0,T]\times  \Omega\times H  \longrightarrow H$ and $g:[0,T]\times\Omega\times H  \longrightarrow H$  are both $\mathscr P\times
   \mathscr B(H)/\mathscr B(H) $-measurable
   such that  $b(\cdot,0), g(\cdot,0)\in M_{\mathscr{F}}^2(0,T;H)
   $; the mapping $\sigma:[0,T]\times  \Omega \times H  \longrightarrow l^2(H) $ is $\mathscr P\times
   \mathscr B(H)/\mathscr B(l^2(H)) $-measurable
   such that $\sigma(\cdot, 0)\in l_{\mathscr{F}}^2(0, T, H)$.
   And there exists a constant $C$  such that
   for all $x, \bar x\in V$ and a.s.$(t,\omega)\in [0,T]\times \Omega,$
   \begin{eqnarray}
     \begin{split}
       ||b(t,x)-b(t,x)||_H+ ||g(t,x)-g(t,x)||_H
       +||\sigma(t, x)-\sigma(t,x)||
       _{l^2(H)} \leq  C||x-\bar x||_H.
     \end{split}
   \end{eqnarray}

\end{ass}

\begin{defn}
\label{defn:c1}
A $V$-valued, $\{{\mathscr F}_t\}_{0\leq t\leq T}$-adapted process $X(\cdot)$ is said to be a solution to the
SEE \eqref{eq:3.1}, if $X (\cdot) \in { M}_{
\mathscr F}^2 ( 0, T; V )$ such that for every $\phi \in V$
and a.e. $( t, \omega ) \in [0, T ] \times \Omega$, it holds that
\begin{eqnarray}
\left\{
\begin{aligned}
( X (t), \phi )_H =& \ ( x, \phi )_H + \int_0^t \left < A (s) X (s), \phi \right > d s
+\int_0^t ( b ( s, X (s)), \phi )_H d s + \int_0^t ( B(s)X(s)+ g ( s, X (s) ), \phi )_H d W (s)
\\& + \sum_{i=1}^{\infty}\int_0^t (\sigma^i ( s,X (s-) ), \phi )_H dH^i(s), \quad t \in [ 0, T ] , \\
X(0) =& \ x\in H  ,
\end{aligned}
\right.
\end{eqnarray}
or alternatively, $X (\cdot)$ satisfies the following It\^o's equation in $V^*$:
\begin{eqnarray}
\left\{
\begin{aligned}
X (t)=& \  x+ \int_0^t  A (s) X (s)d s
+\int_0^t  b ( s, X (s))d s + \int_0^t
[B(s)X(s)+ g ( s, X (s) )] d W (s)
 + \sum_{i=1}^{\infty}\int_0^t \sigma^i ( s, X (s-) ) d H^i(s), \\
X(t) =& \ x\in H .
\end{aligned}
\right.
\end{eqnarray}
\end{defn}

Now we state our main result.

\begin{thm} \label{thm:3.1}
  Let Assumptions \ref{ass:3.1}-\ref{ass:3.2} be
  satisfied by any given coefficients
  $(A,B,b,g,\sigma)$ of the SEE \eqref{eq:3.1}. Then  for any initial
    value $X(0)=x,$ the
    SEE \eqref{eq:3.1} has a unique
  solution $X(\cdot)\in M_{\mathscr{F}}^2(0,T;V).$
\end{thm}
To prove this  theorem, we first show the following
result on  the continuous dependence of  the solution to  the  SEE \eqref{eq:3.1}.

 \begin{thm} \label{thm:3.2}
  Let  $ X(\cdot)$  be a solution to
  the  SEE   \eqref{eq:3.1}
   with the initial value $X(0)=x$ and
    the coefficients $(A,B, b,g,\sigma)$
   which satisfy Assumptions \ref{ass:3.1}-\ref{ass:3.2}. Then
  the following estimate holds:
\begin{eqnarray}\label{eq:3.4}
\begin{split}
  \sup_{0\leq t\leq T}\mathbb E[ ||\hat X(t)||_H^2]+{\mathbb E} \bigg [ \int_0^T \| X (t) \|_V^2 d t \bigg ]  \leq K& \bigg \{ ||x||_H^2 + {\mathbb E}
\bigg [ \int_0^T \| b ( t, 0) \|_H^2 d t \bigg ] + {\mathbb E}
\bigg [ \int_0^T \| g ( t, 0) \|_H^2 d t \bigg ]
\\&+ {\mathbb E} \bigg [ \int_{0}^T \| \sigma (t,0) \|^2_{l^2(H)}  d t \bigg ] \bigg \}.
\end{split}
\end{eqnarray}
Furthermore,  suppose that   $ \bar X(\cdot)$  is a solution to
  the  SEE   \eqref{eq:3.1}
   with the initial value $\bar X(0)=\bar x \in H$ and the coefficients $(A,B,  \bar b, \bar g,\bar\sigma)$
   satisfying Assumptions \ref{ass:3.1}-\ref{ass:3.2},
   then we have

   \begin{eqnarray}\label{eq:3.5}
&& \sup_{0\leq t\leq T}\mathbb E[ || X(t)-\bar X(t)||_H^2]+  {\mathbb E} \bigg [ \int_0^T \| X (t) - {\bar X} (t) \|_V^2 d t \bigg ] \nonumber \\
&& \leq K \bigg \{  \|x-\bar x\|^2_H
+ {\mathbb E} \bigg [ \int_0^T \| b ( t, {\bar X} (t) )
- {\bar b} ( t, {\bar X} (t) ) \|_H^2 d t \bigg ] + {\mathbb E} \bigg [ \int_0^T \| g ( t, {\bar X} (t) ) - {\bar g}( t, {\bar X} (t) ) \|_H^2 d t \bigg ]
\\&&~~~~~~+ {\mathbb E} \bigg [ \int_0^T  \| \sigma ( t,   {\bar X} (t)) - {\bar \sigma}( t,   {\bar X} (t) ) \|_{l^2(H)}^2d t \bigg ]   \bigg \} .\nonumber
\end{eqnarray}
 \end{thm}

 \begin{proof}
   The estimate  \eqref{eq:3.4}
   can be directly obtained by
   the estimate \eqref {eq:3.5}
   by taking  the initial value $\bar X(0)=0$ and the coefficients $(A,B,  \bar b, \bar g,\bar\sigma)=(A, B,0,0,0)$ which imply that
   $\bar X(\cdot)\equiv 0.$
   Therefore, it suffices to prove the estimate
   \eqref {eq:3.5}.
   For the sake of simplicity,
   in the following discussion, we will
   use the following shorthand notation:
   \begin{eqnarray}
   \begin{split}
  & {\hat X} (t) \triangleq X (t) - {\bar X} (t) , \quad \hat x \triangleq x - \bar x,
  \\& \Lambda\triangleq  \|x-\bar x\|^2_H
+ {\mathbb E} \bigg [ \int_0^T \| b ( t, {\bar X} (t) )
- {\bar b} ( t, {\bar X} (t) ) \|_H^2 d t \bigg ] + {\mathbb E} \bigg [ \int_0^T \| g ( t, {\bar X} (t) ) - {\bar g}( t, {\bar X} (t) ) \|_H^2 d t \bigg ]
\\&\quad \quad \quad+ {\mathbb E} \bigg [ \int_0^T  \| \sigma ( t, {\bar X} (t)) - {\bar \sigma}( t, {\bar X} (t) ) \|_{l^2 (H)}^2  d t \bigg ],
\end{split}
\end{eqnarray}
and for
$\phi=b,g,\sigma$
\begin{eqnarray}
\begin{split}
& {\tilde \phi} (t) \triangleq \phi ( t, { X} (t))
- {\bar \phi} ( t, {\bar X} (t) ),
{\hat \phi} (t) \triangleq \phi ( t, {\bar X} (t))
- {\bar \phi} ( t, {\bar X} (t) ) ,
 {\Delta \phi} (t) \triangleq \phi ( t, { X} (t))
- { \phi} ( t, {\bar X} (t) ), ~~t\in [0,T],
\end{split}
\end{eqnarray}
where when  $\phi=\sigma,$
the terms $X(t)$ and
$\bar X(t)$ will be replaced by $X(t-)$
and $\bar X(t-),$ respectively.

Applying It\^{o} formula in Lemma \ref{lem:c1} to $||\hat X(t)||_H^2$
and using Assumptions \ref{ass:3.1}-\ref{ass:3.2} and the elementary inequalities $|a+b|^2\leq 2a^2+2 b^2$
and  $2
a b \leq a^2 +  b^2$, $\forall a, b > 0$, 
 we get  that

\begin{eqnarray}\label{eq:3.6}
     ||\hat X(t)||_H^2=&&||\hat x||^2_H+
     2\int_0^t\langle  A(s)\hat X(s), \hat X(s) \rangle
     ds +2\int_0^t  ( \hat X(s), \tilde b(s))_H ds
+\int_0^t||B(s)\hat X(s)+\tilde g(s)||_H^2ds\nonumber
  \\&&+\sum_{i=1}^\infty
     \sum_{j=1}^\infty\int_0^t (\tilde \sigma^i(s), \tilde\sigma^j(s))_Hd[H^i(s), H^j(s)]
  +2\int_0^t  ( \hat X(s), B(s)\hat X(s)+
     \tilde g(s))_H
     dW(s)\nonumber
     \\&&+2\sum_{i=1}^{\infty}\int_0^t
      (\hat X(s), \tilde \sigma^i(s))dH^i(s)\nonumber
     \\=&&||\hat x||^2_H+
     2\int_0^t\bigg[\langle  A(s)\hat X(s), \hat X(s) \rangle+||B(s)\hat X(s)||_H^2\bigg]
     ds+\int_0^t ||\Delta g(s)+\hat g(s)||_H^2ds\nonumber
     \\&&
     +2\int_0^t  (  B(s)\hat X(s), \Delta g(s)+\hat g(s))_H
     ds+2\int_0^t  ( \hat X(s), \Delta b(s)+\hat b(s))_H
     ds\nonumber
\\&&+\int_0^t || \Delta\sigma(s)+\hat\sigma(s)||_{l^2(H)}
     ds
+\sum_{i=1}^\infty
     \sum_{j=1}^\infty\int_0^t ( \tilde\sigma^i(s), \tilde\sigma^j(s))_H
     d\{[H^i(s), H^j(s)]-\langle H^i(s), H^j(s)\rangle\}
     \\&&+
     2\int_0^t  ( \hat X(s), B(s)\hat X(s)
     +\tilde g(s))_H dW(s)\nonumber
      +2\sum_{i=1}^{\infty}\int_0^t
      (\hat X(s), \tilde \sigma^i(s))dH^i(s)\nonumber
     \\ \leq &&
     K\Lambda-\alpha
\int_0^t   \| \hat X ( s) \|_V^2ds +K
\int_0^t \| \hat X ( s) \|_H^2dt\nonumber
\\&& +\sum_{i=1}^\infty
     \sum_{j=1}^\infty\int_0^t ( \tilde\sigma^i(s), \tilde\sigma^j(s))_H
     d\{[H^i(s), H^j(s)]-\langle H^i(s), H^j(s)\rangle\}\nonumber
     \\&&+
     2\int_0^t  ( \hat X(s), B(s)\hat X(s)
     +\tilde g(s))_H dW(s)\nonumber
      +2\sum_{i=1}^{\infty}\int_0^t
      (\hat X(s), \tilde \sigma^i(s))dH^i(s)\nonumber
 \end{eqnarray}
Taking expectation on both sides of the above
inequality, we  get that

\begin{eqnarray}\label{eq:3.7}
\begin{split}
    \mathbb E[ ||\hat X(t)||_H^2]
    +\alpha\mathbb E\bigg[\int_0^t ||\hat X(s)||_V^2ds\bigg]
    \leq
     K\Lambda
+K{\mathbb E} \bigg [
\int_0^t \| \hat X ( s) \|_H^2dt\bigg].
\end{split}
 \end{eqnarray}
Then by virtue of  Gr\"onwall's inequality to  $\mathbb E[||X(t)||_H^2],$ we obtain
\begin{eqnarray}\label{eq:3.8}
\begin{split}
    \sup_{0\leq t\leq T}\mathbb E[ ||\hat X(t)||_H^2]
    +\mathbb E\bigg[\int_0^T ||\hat X(s)||_V^2ds\bigg]
    \leq &
     K \Lambda.
\end{split}
 \end{eqnarray}
The proof is complete.
 \end{proof}
In the following, we give the existence and
uniqueness result for  the solution
of the SEE \eqref{eq:3.1}
for a simple case where the
coefficients $(b, g, \sigma)$ is
independent of the variable $x.$

\begin{lem} \label{lem:3.3}
  Given three stochastic processes  $b,g$ and $\sigma$  such that $b(\cdot)\in M_{\mathscr{F}}^2(0,T;H), g(\cdot)\in M_{\mathscr{F}}^2(0,T;H)$ and
   $\sigma(\cdot)\in M_{\mathscr{F}}^2(0,T;l^2(H)).$
 Suppose that the operators $ A$ and $B$  satisfy  Assumption \ref{ass:3.1}.
  Then there exists a unique solution
  $X(\cdot)\in M_{\mathscr{F}}^2(0,T;V)$ to the
  following SEE:
\begin{eqnarray} \label{eq:3.12}
  \left\{
  \begin{aligned}
   d X (t) = & \ [ A (t) X (t) + b ( t) ] d t
+ [B(t)X(t)+g( t) ]d W(t)+ 
\sum_{i=1}^{\infty}\sigma^i (t )dH^i(t),  \\
X (0) = & \  x , \quad t \in [ 0, T ].
  \end{aligned}
  \right.
\end{eqnarray}
\end{lem}

\begin{proof}
The  proof can be obtained by Galerkin approximations  in the same way as
  the proof of  Theorem 3.2 in  \cite{chenshaokuan}
  with minor change. Now we begin our 
  proof.
First of all, we fix a standard complete orthogonal basis $\{e_{i} | i=1,2,3,\dots\}$ in the space $H$ which is dense  in the space $V$.
For any $n,$ consider the following
finite-dimensional stochastic differential
equation in $\mathbb R^n:$
\begin{eqnarray}\label{eq:6.18}
\left\{
\begin{aligned}
x_1^n(t)=& \ (x, e_1)_H +\int_0^t \bigg [ \sum_{j=1}^nx_j^n(s) \langle A (s) e_j, e_1\rangle
+ (b(s), e_1)_H \bigg ] d s +\int_0^t \bigg[\sum_{j=1}^n x_j^n(s) (B(s)e_j, e_1)_H + (g(s), e_1)_H\bigg]d W (s)
\\&+\sum_{i=1}^{\infty}\int_0^t(\sigma^i(s), e_1)_Hd H^i (s) , \\
x_2^n(t)=& \ (x, e_2)_H+\int_0^t \bigg [ \sum_{j=1}^nx_j^n(s) \langle A (s) e_j, e_1\rangle
+ (b(s), e_2)_H \bigg ] d s +\int_0^t \bigg[\sum_{j=1}^n x_j^n(s) ( B(s)e_j, e_2)_H + (g(s), e_2)_H\bigg]d W (s)
\\&+\sum_{i=1}^{\infty}\int_0^t(\sigma^i(s), e_2)_Hd H^i (s) , \\
&\vdots \\
x_n^n(t)=& \ (x, e_n)_H+\int_0^t \bigg [ \sum_{j=1}^nx_j^n(s) \langle A (s) e_j, e_n\rangle
+ (b(s), e_n)_H \bigg ] d s +\int_0^t \bigg[\sum_{j=1}^n x_j^n(s)  (B(s)e_j, e_n)_H + (g(s), e_n)_H\bigg]d W (s)
\\&+\sum_{i=1}^{\infty}\int_0^t(\sigma^i(s), e_n)_Hd H^i (s) ,
\end{aligned}
\right.
\end{eqnarray}
Under Assumptions \ref{ass:3.1}
and \ref{ass:3.2}, from the existence and uniqueness theory for
the finite dimensional SDE driven
by Teugels martingale , the above equation admits a unique  strong solution
$x^n(\cdot) \in{ M}_{
\mathscr F}^2 ( 0, T; \mathbb R^n )$ ,
where $x^n(\cdot)=(x_1^n(\cdot),\cdots,  x_n^n(\cdot)).$

Now we can define an  approximation solution to \eqref{eq:3.12} as follows:
\begin{eqnarray*}
X^n(t):=\sum_{i=1}^{n}x_{i}^{n}(t)e_{i}, 
\end{eqnarray*}
where \begin{eqnarray*}
X^n(0):=\sum_{i=1}^{n}(x,e_{i})_He_i.
\end{eqnarray*}
Then the equation \eqref{eq:6.18} can be written as
\begin{eqnarray}\label{eq:6.21}
(X^{n}(t),e_{i})_{H} &=&\left( X^n(0),e_{i}\right)_{H}
+\int_{0}^{t}\bigg[\left\langle A( s) X^{n}(s),e_{i}\right\rangle
 +( b(s),e_{i})_H\bigg]ds \nonumber \\&&+\int_{0}^{t}\bigg[
 (B(s) X^{n}(s),e_{i})
+\left( g(s),e_{i}\right) _{H}\bigg]dW(s)
+\sum_{i=1}^{\infty}\int_0^t \int_0^t(\sigma^i(s), e_n)_Hd H^i (s) , \quad i=1,\cdots, n .
\end{eqnarray}
Now applying It\^{o} formula to $||X^n(t)||^2_H$, we get  that
\begin{eqnarray}\label{eq:6.2}
|| X^n(t)||_H^2=&&||\hat X^n(0)||^2_H+
     2\int_0^t\langle  A(s)X^n(s), X^n(s) \rangle
     ds +2\int_0^t  (X^n(s),  b(s))_H ds
+\int_0^t||B(s)X^n(s)+ g(s)||_H^2ds\nonumber
  \\&&+\sum_{i=1}^\infty
     \sum_{j=1}^\infty\int_0^t ( \sigma^i(s),\sigma^j(s))_Hd[H^i(s), H^j(s)]
  +2\int_0^t  ( X^n(s), B(s) X^n(s)+
      g(s))_H
     dW(s)\nonumber
     \\&&+2\sum_{i=1}^{\infty}\int_0^t
      ( X^n(s), \sigma^i(s))dH^i(s)\nonumber .
\end{eqnarray}
Therefore, under Assumptions \ref{ass:3.1}
and \ref{ass:3.2}, similar to the proof of the estimate \eqref{eq:3.4},  using Gr\"onwall's inequality, we can easily get the following estimate:
\begin{eqnarray}\label{eq:3.17}
\begin{split}
&{\mathbb E} \bigg [ \int_0^T \| X^n (t) \|_V^2 d t \bigg ]  \leq K \bigg \{ ||x||_H^2 + {\mathbb E}
\bigg [ \int_0^T \| b ( t) \|_H^2 d t \bigg ] + {\mathbb E}
\bigg [ \int_0^T \| g ( t) \|_H^2 d t \bigg ]
+ {\mathbb E} \bigg [ \int_{0}^T \| \sigma (t) \|^2_{l^2(H)}  d t \bigg ] \bigg \}.
\end{split}
\end{eqnarray}
This inequality implies that there is a subsequence $\{n^\prime\}$ of $\left\{ n\right\} $ and a triple $X(\cdot)
\in { M}_{\mathscr F}^2 ( 0, T; V )$ such that
\begin{eqnarray}\label{eq:3.18}
& X^{n^{\prime }}\rightarrow X\text{ weakly in }{ M}_{\mathscr F}^2 ( 0, T; V ).
\end{eqnarray}
Let $\Pi$ be an arbitrary bounded random variable on $(\Omega,\mathscr F)$ and
$\psi$ be an arbitrary bounded measurable function on $[0,T]$. From  the equality \eqref{eq:6.21}, for $n^\prime\in \mathbb{N}^*$
and basis $e_{i}$, where $i\leq n^{\prime}$, we have
\begin{eqnarray}\label{eq:6.25}
&& \mathbb E \bigg [ \int_0^T \Pi \psi(t) (X^{n^\prime}(t),e_{i})_{H}dt \bigg ] \nonumber \\
&& = \mathbb E \bigg [ \int_0^T \Pi \psi(t) \bigg\{\left( X^{n^\prime}(0),e_{i}\right)_{H}
+\int_{0}^{t}\bigg[\left\langle A( s) X^{n^\prime}(s),e_{i}\right\rangle
 +( b(s),e_{i})_H\bigg]ds \nonumber \\&&+\int_{0}^{t}\bigg[
 (B(s) X^{n^\prime}(s),e_{i})
+\left( g(s),e_{i}\right) _{H}\bigg]dW(s)
+\sum_{i=1}^{\infty}\int_0^t \int_0^t(\sigma^i(s), e_{i})_Hd H^i (s)\bigg\}dt \bigg ] .
\end{eqnarray}
Now letting $n'\longrightarrow \infty$ on the both sides of the above equation to get its limit. Firstly, from the weak convergence
property of $\{X^n\}_{n=1}^\infty$ in ${\cal M}_{\cal F}^2 ( 0, T; V )$, we have
\begin{eqnarray}\label{eq:6.26}
    \lim_{n'\rightarrow \infty}\mathbb E \bigg [ \int_{0}^{T}\Pi\psi(t)(X^{n'}(t), e_i)_H dt \bigg ]
    &=&\lim_{n'\rightarrow \infty}\mathbb E \bigg [ \int_{0}^{T}\mathbb E[\Pi|\mathscr F_t]\psi(t)(X^{n'}(t), e_i)_H dt \bigg ] \nonumber \\
    &=&\lim_{n'\rightarrow \infty}\mathbb E \bigg [ \int_{0}^{T}(X^{n'}(t),\mathbb E[\Pi|\mathscr F_t]\psi(t)e_i)_H dt \bigg ]  \nonumber \\
    &=&\mathbb E \bigg [ \int_{0}^{T}(X(t),\mathbb E[\Pi|\mathscr F_t]\psi(t)e_i)dt \bigg ] \nonumber \\
    &=&\mathbb E \bigg [ \int_{0}^{T}\Pi\psi(t)(X(t),e_i)dt \bigg ] ,
\end{eqnarray}
and
\begin{eqnarray}
    \lim_{n'\rightarrow \infty}\mathbb E \bigg [ \int_{0}^{t}\Pi\langle A(s)X^{n'}(s), e_i\rangle ds \bigg ]
    &=&\lim_{n'\rightarrow \infty}\mathbb E \bigg [ \int_{0}^{t}\mathbb E[\Pi|\mathscr F_s]\langle A(s)X^{n'}(s), e_i\rangle ds \bigg ] \nonumber \\
    &=&\lim_{n'\rightarrow \infty}\mathbb E \bigg [ \int_{0}^{t}\langle A(s)X^{n'}(s),\mathbb E[\Pi|\mathscr F_s]e_i\rangle ds \bigg ] \nonumber \\
    &=&\lim_{n'\rightarrow \infty}\mathbb E \bigg [ \int_{0}^{t}\langle X^{n'}(s), A^*(s)\mathbb E[\Pi|\mathscr F_s]e_i\rangle ds \bigg ] \nonumber \\
    &=&\mathbb E \bigg [ \int_{0}^{t}\langle X(s), A^*(s)\mathbb E[\Pi|\mathscr F_s]e_i\rangle ds \bigg ] \nonumber \\
    &=&\mathbb E \bigg [ \int_{0}^{t}\Pi\langle A(s)X(s),e_i\rangle ds \bigg ] .
\end{eqnarray}
In view of \eqref{eq:3.3} and \eqref{eq:3.17}, we conclude that
\begin{eqnarray*}
\mathbb E \bigg [ \bigg| \int_{0}^{t} \Pi \la A(s)X^{n'}(s), e_i\rangle ds\bigg| \bigg ]
\leq C \bigg \{ \mathbb E \bigg [ \int_{0}^{T} ||X^{n'}(s)||_V^2ds \bigg ] \bigg\}^{\frac{1}{2}} <C<\infty,
\end{eqnarray*}
where the constant $C$ is independent of $n'$. Hence from Fubini's Theorem and Lebesgue's Dominated Convergence Theorem, we have
\begin{eqnarray}\label{eq:6.28}
\lim_{n'\longrightarrow \infty}\mathbb E \bigg [ \int_{0}^{T}\Pi\psi(t)\int_{0}^{t}\langle A(s)X^{n'}(s),e_i\rangle ds dt \bigg ] \nonumber 
 =&& \lim_{n'\longrightarrow \infty}\int_{0}^{T} \psi(t) \mathbb E \bigg [ \int_{0}^{t} \Pi\langle A(s)X^{n'}(s),e_i\rangle ds \bigg ] dt
 \nonumber
\\ =&& \int_{0}^{T} \psi(t) \mathbb E \bigg [ \int_{0}^{t} \Pi\langle A(s)X(s),e_i\rangle ds \bigg ] dt\nonumber
\\=&& \mathbb E \bigg [\int_{0}^{T} \psi(t) \int_{0}^{t} \Pi\langle A(s)X(s),e_i\rangle ds \bigg ] dt .
\end{eqnarray}
Similarly, in view of 
\eqref{eq:3.3}, from \eqref{eq:3.17} 
 and \eqref{eq:3.18}, it is easy to 
 check  that 
$$(BX^{n^\prime}, e_i)_H \rightarrow (BX, e_i)_H\text{ weakly in }{ M}_{\mathscr F}^2 ( 0, t; \mathbb R ).$$
Since the stochastic integral with respect to the Brownian motion $W$ are linear and strong continuous mappings from ${ M}_{\mathscr F}^2 ( 0, t; \mathbb R )$ to $L^2(\Omega,{\mathscr F}_T,P;\mathbb R)$,
 it is weakly continuous. Therefore, 
\begin{eqnarray} \label{eq:3.23}
&& \lim_{n' \rightarrow \infty} \mathbb E\bigg[\Pi \int_{0}^{t}\left( B(s)X^{n^\prime}(s),e_{i}\right) _{H}dW(s)\bigg]= \mathbb E \bigg[\Pi \int_{0}^{t}\left( B(s)X(s),e_{i}\right) _{H}dW(s)\bigg] .
 \end{eqnarray}
Moreover, from \eqref{eq:3.17}, we get 
\begin{eqnarray}
&&\psi(t)\mathbb E\bigg[\Pi \bigg(\int_{0}^{t}\left( B(s)X^{n^\prime}(s),e_{i}\right) _{H}dW(s)
\bigg] \nonumber \leq  \frac{1}{2}\psi^2(t)\mathbb E |\Pi|^2
+ C \bigg \{ \mathbb E \bigg [ \int_0^T||B(s)X^{n^\prime}(s)||^2_Hdt \bigg ] \bigg \} \leq C.
\end{eqnarray}
Hence, by Fubini's Theorem and Lebesgue's Dominated Convergence Theorem, we have
\begin{eqnarray}\label{eq:6.33}
&& \lim_{n'\rightarrow \infty} \mathbb E \bigg[ \int_0^T \psi(t)\Pi \bigg(\int_{0}^{t}\left( B(s)X^{n^\prime}(s),e_{i}\right) _{H}dW(s)
\bigg]dt \nonumber \\
&&=\lim_{n'\rightarrow \infty} \int_0^T \psi(t) \mathbb E\bigg[\Pi \bigg(\int_{0}^{t}\left( B(s)X^{n^\prime}(s),e_{i}\right) _{H}dW(s)
\bigg] d t \nonumber \\
&&=\int_t^T\psi(t)\mathbb E
      \bigg[\Pi \bigg(\int_{0}^{t}\left( B(s)X(s),e_{i}\right) _{H}dW(s)\bigg)\bigg]dt
\\&&=\mathbb E\int_t^T
      \bigg[\psi(t)\Pi \bigg(\int_{0}^{t}\left( B(s)X(s),e_{i}\right) _{H}dW(s)\bigg)\bigg]dt .
 \end{eqnarray}

Therefore, combining \eqref{eq:6.26},\eqref{eq:6.28}, \eqref{eq:3.23} and \eqref{eq:6.33}, and letting $n'\rightarrow \infty$ in \eqref{eq:6.25},
we can conclude that

\begin{eqnarray}\label{eq:3.26}
&& \mathbb E \bigg [ \int_0^T \Pi \psi(t) (X(t),e_{i})_{H}dt \bigg ] \nonumber \\
&& = \mathbb E \bigg [ \int_0^T \Pi \psi(t) \bigg\{\left( X(0),e_{i}\right)_{H}
+\int_{0}^{t}\bigg[\left\langle A( s) X(s),e_{i}\right\rangle
 +( b(s),e_{i})_H\bigg]ds \nonumber \\&&+\int_{0}^{t}\bigg[
 (B(s) X(s),e_{i})
+\left( g(s),e_{i}\right) _{H}\bigg]dW(s)
+\sum_{i=1}^{\infty}\int_0^t \int_0^t(\sigma^i(s), e_{i})_Hd H^i (s)\bigg\}dt \bigg ] .
\end{eqnarray}
This implies that for a.s. $\left( t,\omega \right) \in \lbrack 0,T]\times \Omega$,
\begin{eqnarray}\label{eq:6.34}
(X(t),e_{i})_{H} &=&\left( X(0),e_{i}\right)_{H}
+\int_{0}^{t}\bigg[\left\langle A( s) X(s),e_{i}\right\rangle
 +( b(s),e_{i})_H\bigg]ds \nonumber \\&&+\int_{0}^{t}\bigg[
 (B(s) X(s),e_{i})
+\left( g(s),e_{i}\right) _{H}\bigg]dW(s)
+\sum_{i=1}^{\infty}\int_0^t \int_0^t(\sigma^j(s), e_i)_Hd H^j (s),
\end{eqnarray}
thanks to the arbitrariness of $\Pi$ and $\psi(\cdot)$.
Since the standard complete orthogonal basis $\{e_{i} | i=1,2,3,\dots\}$ in $H$ is dense in the space $V$, for every $\phi \in V$ and
a.e. $( t, \omega ) \in [ 0, T ] \times \Omega$, it holds that
\begin{eqnarray}\label{eq:6.34}
(X(t),\phi)_{H} &=&\left( X(0),\phi\right)_{H}
+\int_{0}^{t}\bigg[\left\langle A( s) X(s),\phi\right\rangle
 +( b(s),\phi)_H\bigg]ds \nonumber \\&&+\int_{0}^{t}\bigg[
 (B(s) X(s),\phi)
+\left( g(s),\phi\right) _{H}\bigg]dW(s)
+\sum_{i=1}^{\infty}\int_0^t \int_0^t(\sigma^j(s), \phi)_Hd H^j (s) .
\end{eqnarray}
Therefore, from the Definition \ref{defn:c1}, we conclude that the triple $ X(\cdot)$
is the solution to the SEE \eqref{eq:3.12}. Thus the existence is proved. For the uniqueness, the 
proof can be directly by the priori 
estimate \eqref{eq:3.5}. The
proof is complete.
\end{proof}

{ \bf\emph{Proof of Theorem \ref{thm:3.1}.}} The uniqueness of the solution to the SEE
\eqref{eq:3.1} can be got   by the a priori estimate
\eqref{eq:3.5} directly.
  For  $\rho \in [0, 1] $ and any three   given stochastic processes
   $b_0 \in { M}_{\mathscr F}^2 ( 0, T; H ),$
$g_0 \in {M}_{\mathscr F}^2 ( 0, T; H ),$  and  $\sigma_0 \in {M}_{\mathscr F}^2 ( 0, T; l^2(H) )$ we introduce  a family of parameterized  SEEs as follows:
\begin{eqnarray}\label{eq:3.30}
\begin{split}
 X(t) =& \ x +\int_0^t A (s) X (s) d s + \int_0^t \Big [\rho  b ( s, X(s) )]+ b_0(s)\Big]d s+\int_0^t \Big[ B(s)X(s)+\rho g ( s, X ( s ) )+g_0(s)\Big ] d W (s)
 \\&+\sum_{i=1}^\infty\int_0^t \Big[\rho \sigma^i (s,  X(s))
 +\sigma_0^i(s )\Big]  dH^i(s).
\end{split}
\end{eqnarray}
It is easy to see that
when we take the parameter $\rho=1$ and
$b_0 \equiv 0,g_0\equiv0,
\sigma_0 \equiv0,$
 the  SEE \eqref{eq:3.30}
 is reduced to  the original SEE
\eqref{eq:3.1}. Obviously,  the  coefficients of the SEE \eqref{eq:3.30}
satisfy  Assumption \ref{ass:3.1} and
\ref{ass:3.2}
with $(A, B,  b ,  g, \sigma)$  replaced by $(A, B, \rho b + b_0, \rho g+ g_0, \rho \sigma+ \sigma_0)$.
 Suppose
 for any
$b_0  \in { M}_{\mathscr F}^2 ( 0, T; H ),$
$g_0 \in { M}_{\mathscr F}^2 ( 0, T; H ),$  $\sigma_0 \in { M}_{\mathscr F}^2 ( 0, T; l^2(H) ) $ and  some  parameter $\rho = \rho_0$,
 there exists a unique solution $X (\cdot)\in { M}_{\mathscr F}^2 ( 0, T; V)$ to the SEE  \eqref{eq:3.30}.  For any parameter $\rho$,
the SEE \eqref{eq:3.30}
can be rewritten as
\begin{eqnarray}\label{eq:3.14}
\begin{split}
 X(t) =& \ x + \int_0^t A (s) X (s) d s + \int_0^t \Big [\rho_0 b ( s,X(s) )+ b_0(s)+(\rho-\rho_0)  b ( s,  X(s) )\Big]d s
 \\&+\int_0^t \Big[ B(s)X(s)+\rho_0g ( s, X ( s ) )+g_0(s)
 +(\rho-\rho_0)g ( s, X ( s ) )\Big ] d W (s)
 \\&+\su\int_0^t \Big[\rho_0
  \sigma^i( s,  X ( s ) )+\sigma^i_0(t )
 +(\rho-\rho_0)\sigma^i( s,  X ( s ) )\Big ] dH^i(s).
\end{split}
\end{eqnarray}
 Therefor  by the above assumption,
  for any $x (\cdot) \in {M}_{\mathscr F}^2 ( 0, T; V),$
the following  SEE
\begin{eqnarray}\label{eq:3.19}
\begin{split}
 X(t) =& \ x + \int_0^t A (s) X (s) d s + \int_0^t \Big [\rho_0 b ( s,X(s) )+ b_0(s)+(\rho-\rho_0)  b ( s,  x(s) )\Big]d s
 \\&+\int_0^t \Big[ B(s)X(s)+\rho_0g ( s, X ( s ) )+g_0(s)
 +(\rho-\rho_0)g ( s, x ( s ) )\Big ] d W (s)
 \\&+\su\int_0^t \Big[\rho_0
  \sigma^i( s,  X ( s ) )+\sigma^i_0(s )
 +(\rho-\rho_0)\sigma^i( s,  x ( s ) )\Big ] dH^i(s)
\end{split}
\end{eqnarray}
 admits a unique solution $X (\cdot) \in { M}_{\mathscr F}^2 ( 0, T; V)$.
Now define  a mapping from $ { M}_{\mathscr F}^2 ( 0, T; V)$ onto itself denoted by $$X (\cdot) = \Gamma (x (\cdot)).$$
Then for any $x_i (\cdot) \in { M}_{\mathscr F}^2 ( 0, T; V)$, $i = 1, 2$, from the Lipschitz
continuity of $b$, $g$, $\sigma$ and  a priori estimate \eqref{eq:3.5},
it follows that
\begin{eqnarray*}
|| \Gamma ( x_1 (\cdot) ) - \Gamma ( x_2 (\cdot) ) ||^2_{{ M}_{\mathscr F}^2 ( 0, T; V)}
&=& || X_1 (\cdot) - X_2(\cdot) ||^2_{{ M}_{\mathscr F}^2 ( 0, T; V)}\\
&\leq& K |\rho-\rho_0|^2 \cdot || x_1 (\cdot) - x_2(\cdot) ||^2_{{ M}_{\mathscr F}^2 ( 0, T; V)}.
\end{eqnarray*}
Here $K $ is a  positive constant independent of $\rho$. If $| \rho - \rho_0 |< \frac{1}{2 \sqrt {K}}$, the mapping $\Gamma$ is strictly contractive in
${ M}_{\mathscr  F}^2 ( 0, T; V )$. Hence it implies that
the SEE \eqref{eq:3.12} with the coefficients $( A, B, \rho b + b_0, \rho g + g_0, \rho \sigma + \sigma_0)$ admits a unique solution
$X (\cdot)\in {M}_{\cal F}^2 ( 0, T; V )$.
From Lemma \ref{lem:3.3},  the uniqueness and existence of the solution to the SEE \eqref{eq:3.12} is true for $\rho=0$. Then starting from $\rho = 0$,  one can reach $\rho = 1$ in finite steps and this finishes
 the proof of solvability of the SEE \eqref{eq:3.1}. This completes the proof.

\section{ Formulation of  Optimal Control Problem }
Let $U$  be a real-valued Hilbert space standing for the control space. Let  ${\mathscr U}$  be a nonempty  convex
closed subset of $U$.  An  admissible control process $u (\cdot) \triangleq \{ u (t), 0 \leq t \leq T \}$ is  defined
as follows.
\begin{defn}
A stochastic  process $u (\cdot)$
 defined on $[0, T]\times \Omega$ is called an admissible control process  if it is a predictable  process such that
$u (\cdot) \in { M}^2_{\mathscr F} ( 0, T; U )$ and $u (t) \in {\mathscr U}$, a.e. $t \in [0, T]$, ${\mathbb P}$-a.s..
Write ${\cal A}$ for the set of all admissible control processes.
\end{defn}
In the Gelfand triple $( V, H, V^* ),$
for any admissible control $u(\cdot)\in \cal A,$
we consider  the following  controlled
SEE driven by Teugels martingales
\begin{eqnarray} \label{eq:4.1}
  \left\{
  \begin{aligned}
   d X (t) = & \ [ A (t) X (t) + b ( t, X (t), u(t)) ] d t
+ [B(t)X(t)+g( t, X (t), u(t)) ]d W(t)
 \\&\quad + \su \sigma^i (t, X(t-),u(t)) dH^i(t),  \\
X (0) = & \  x , \quad t \in [ 0, T ]
  \end{aligned}
  \right.
\end{eqnarray}
with the cost functional
\begin{eqnarray}\label{eq:4.2}
J ( u (\cdot) ) = {\mathbb E} \bigg [ \int_0^T l ( t, x (t), u (t) ) d t
+ \Phi ( x (T) ) \bigg ].
\end{eqnarray}
where the coefficients
satisfy the following
basic assumptions:
\begin{ass}\label{ass:2.5}
\begin{enumerate}
\item[]
\item[(i)]

$A:[0,T]\times \Omega \longrightarrow {\mathscr L} (V, V^*)$
and
$B: [0,T]\times \Omega \longrightarrow {\mathscr L} (V, H)$
are operator-valued stochastic processes satisfying (i) in Assumption \ref{ass:3.1};
\item[(iii)]$b, g: [ 0, T ] \times \Omega \times H \times {\mathscr U} \rightarrow H$  are $\mathscr P\times
   \mathscr B(H)\times \mathscr B(\mathscr U)/\mathscr B(H) $ measurable
   mappings and $\sigma=(\sigma)_{i=1}^\infty:[0,T]\times  \Omega \times H \times \mathscr U \longrightarrow l^2(H) $ is a $\mathscr P\times
   \mathscr B(H)\times \mathscr B(U)/\mathscr B(l^2(H))$-measurable mapping  such that $b ( \cdot, 0, 0 ), g ( \cdot, 0, 0 ) \in {
M}^2_{\mathscr F} ( 0, T; H ), \sigma(\cdot, 0,0) \in { M}_{\mathscr F}^2 ( 0, T; l^2(H) ).$ Moreover, for almost all $( t, \omega) \in [ 0, T ] \times \Omega $,  $b$, $g$ and $\sigma$ are  G\^ateaux differentiable in $(x,u)$ with  continuous bounded G\^ateaux  derivatives
$b_x, g_x,\sigma_x,  b_u, g_u$ and  $\sigma_u$;
\item[(iv)]
$l:[ 0, T ] \times \Omega \times H \times {\mathscr U} \rightarrow \mathbb R $ is a
 ${\mathscr P} \otimes {\mathscr B} (H) \otimes {\mathscr B} ({\mathscr U})/ {\mathscr B} ({\mathbb R})$-measurable mapping  and  $\Phi: \Omega \times H \rightarrow {\mathbb R}$
is  a ${\mathscr F}_T\otimes {\mathscr B} (H) / {\mathscr B} ({\mathbb R})$-measurable
mapping.
For almost all $( t, \omega ) \in [ 0, T ] \times \Omega$, $l$ is continuous G\^ateaux
differentiable in $(x,u)$
with continuous  G\^ateaux derivatives $l_x$ and $l_u$, and $\Phi$
is G\^ateaux differentiable in $x$
with  continuous
 G\^ateaux derivative $\Phi_x$.
Moreover, for almost all $( t, \omega ) \in [ 0, T ] \times \Omega$,   there exists a   constant $C > 0$ such that  for all $( x, u ) \in H  \times {\mathscr U}$
\begin{eqnarray*}
| l ( t, x, u  ) |
\leq  C ( 1 + \| x \|^2_H + + \| u \|_U^2 ) ,
\end{eqnarray*}
\begin{eqnarray*}
&& \| l_x ( t, x,u) \|_H +
+ \| l_u ( t, x, u ) \|_U \leq C ( 1 + \| x \|_H  + \| u \|_U  ) ,
\end{eqnarray*}
and
\begin{eqnarray*}
& | \Phi (x) | \leq C ( 1 + \| x \|^2_H) , \\
& \| \Phi_x (x) \|_H \leq C ( 1 + \| x \|_H).
\end{eqnarray*}
\end{enumerate}
\end{ass}

For any admissible control $u(\cdot),$ the solution of the system { \eqref{eq:4.1}},  denoted by $X^u(\cdot)$ or $X(\cdot),$
  if its dependence on 
   $u(\cdot)$  is clear from  the context,  is called the
state process corresponding to the control process $u(\cdot)$, and
 $(u(\cdot); X(\cdot))$ is called an
admissible pair.
The following result gives the well-posedness of the state equation as well as some useful estimates.

\begin{lem}\label{lem:1.1}
Let Assumption \ref{ass:2.5}  be satisfied. Then for any admissible control $u(\cdot)$,
the state equation
\eqref{eq:4.1}
has a unique solution $X^u(\cdot) \in M_{\mathscr{F}}^2(0,T;V).$
Moreover,  the following estimate holds

\begin{eqnarray}\label{eq:4.3}
\begin{split}
&\sup_{0\leq t\leq T}\mathbb E[ ||\hat X^u(t)||_H^2]+{\mathbb E} \bigg [ \int_0^T \| X^u (t) \|_V^2 d t \bigg ] \leq K \bigg \{ 1+||x||_H^2+ {\mathbb E}
\bigg [ \int_0^T \| u ( t) \|_U^2 d t \bigg ]\bigg \}
\end{split}
\end{eqnarray}
and
\begin{eqnarray}\label{eq:4.4}
\begin{split}
  |J( u(\cdot))|< \infty.
  \end{split}
\end{eqnarray}
Furthermore,  let $ X^v(\cdot)$  be
the state process corresponding to
another admissible control $v(\cdot),$
then

   \begin{eqnarray}\label{eq:4.5}
&&\sup_{0\leq t\leq T}\mathbb E[ ||\hat X^u(t)-X^v(t)||_H^2]+ {\mathbb E}
\bigg [ \int_0^T \| X^u (t) - { X}^v (t) \|_V^2 d t \bigg ] \leq K {\mathbb E} \bigg [ \int_0^T \| u(t) -v(t) \|_U^2 d t \bigg ] .
\end{eqnarray}
\end{lem}

\begin{proof}
  Under Assumption \ref{ass:2.5}, by Theorem \ref{thm:3.1}, we  can get directly
  the existence and uniqueness of the  solution of the state equation \eqref{eq:3.1}.  And
  the estimates \eqref{eq:4.3} and \eqref{eq:4.5} can be obtained by
  the estimates \eqref{eq:3.4} and \eqref{eq:3.5}, respectively.
  Furthermor  from Assumption
   \ref{ass:2.5} and  the  estimate  \eqref{eq:4.3}, it follows that

   \begin{eqnarray}\label{eq:1.7}
     \begin{split}
       |J(u(\cdot))|\leq K \bigg\{\sup_{0\leq t\leq T}\mathbb E \bigg[|| X(t)||_H^2 \bigg] +\mathbb E\bigg[\int_0^T
  ||u(t)||^2_U dt\bigg]+1\bigg\} \leq K \bigg\{1+||x||_H^2+\mathbb E\bigg[\int_0^T
  ||u(t)||_U^2 dt\bigg]\bigg\}<\infty.
     \end{split}
   \end{eqnarray}
The proof is complete.
\end{proof}
Therefor  by Lemma \ref{lem:1.1}, we claim  that the cost functional \eqref{eq:4.2}
is well-defined.
Our optimal control problem can be stated as follows.
\begin{pro}\label{pro:2.1}
Find an admissible control process ${\bar u} (\cdot) \in {\cal A}$ such that
\begin{eqnarray}\label{eq:b8}
J ( {\bar u}(\cdot) ) = \inf_{u (\cdot) \in {\cal A}} J ( u (\cdot) ) .
\end{eqnarray}
\end{pro}
The admissible control ${\bar u} (\cdot)$ satisfying (\ref{eq:b8}) is called an optimal control process of
Problem \ref{pro:2.1}. Correspondingly, the state process ${\bar X} (\cdot)$ associated with ${\bar u} (\cdot)$
is called an optimal state process. Then $( {\bar u} (\cdot); {\bar X} (\cdot) )$ is called an optimal pair of
Problem \ref{pro:2.1}.

\section{Regularity Result for the Adjoint Equation}

For any admissible pair $( \bar u (\cdot); \bar X(\cdot) )$, the corresponding  adjoint
processes is defined as  
a triple $(\bar p(\cdot), \bar q(\cdot), \bar r(\cdot))$ of
 stochastic processes,  which is
a solution to
the following
backward stochastic evolution equation
(BSEE for short)  driven by Teugels martingales,  called the
adjoint equation,
\begin{eqnarray}\label{eq:4.4}
\begin{split}
   \left\{\begin{array}{ll}
d\bar p(t)=&-\bigg[A^*(t)\bar p(t)+b_x^*(t, \bar X(t), \bar u(t))\bar p(t)
+B^*(t)\bar q(t)+g_x^{*}(t, \bar X(t), \bar u(t))\bar q(t)
\\&\quad\quad+\displaystyle\su
\sigma{^i}_x^{*}(t, \bar
X(t),\bar u(t))\bar r^i(t)
+l_x(t, \bar X(t), \bar u(t))\bigg]dt
\\&\quad\quad+\bar q(t)dW(t)+\displaystyle\su\bar r^i(t)dH^i(t),
~~~~0\leqslant t\leqslant T,
\\ \bar p(T)=&\Phi_x(\bar X(T)).
  \end{array}
 \right.
 \end{split}
  \end{eqnarray}
Here $A^*$ denotes the adjoint operator of
the operator $A.$ Similarly, we can define
the corresponding adjoint operator for other coefficients.\\
Under Assumptions \ref{ass:2.5}, we have the following
basic result for the adjoint process.

\begin{lem}
 Let Assumptions \ref{ass:2.5}
 be satisfied. Then for any admissible pair
 $( \bar u (\cdot); \bar X (\cdot) ),$
  there exists a unique adjoint process
  $(\bar p(\cdot), \bar q(\cdot), \bar r(\cdot))\in M_{\mathscr{F}}^2(0,T;V)
  \times M_{\mathscr{F}}^2(0,T;H)\times M_{\mathscr{F}}^2(0,T;l^2(H)).$
  Moreover, the following estimate holds:

  \begin{eqnarray}
    \begin{split}
    &{\mathbb E} \bigg [ \int_0^T \|\bar p (t) \|_V^2 d t \bigg ] +{\mathbb E} \bigg [ \int_0^T \| \bar q (t) \|_H^2 d t \bigg ]
+ {\mathbb E} \bigg [ \int_0^T
 \| r (t ) \|_{l^2(H)}^2  d t \bigg ]
\leq  K\bigg\{  {\mathbb E} \bigg [ \int_0^T \| l_x (t,\bar X(t),\bar u(t)) \|_H^2 d t \bigg ]+
\mathbb E[||\Phi_x(\bar X(T))||_H^2]\bigg\}.
    \end{split}
  \end{eqnarray}
\end{lem}

\begin{proof}
From the property of adjoint operator,
 the adjoint operator $A^*$  of $A$  and the
  adjoint operator $B^*$ of $B$
 also  satisfies (i) in Assumption \ref{ass:3.1}. Therefore, similarly
 to Theorem \ref{thm:3.1}, the 
 existence and uniqueness of the 
 solution can be proved  by Galerkin approximations and
 parameter extension method.
\end{proof}
Define the Hamiltonian ${\cal H}: [ 0, T ] \times \Omega \times H  \times {\mathscr U} \times H \times H \times  l^2(H)
\rightarrow {\mathbb R}$ by
\begin{eqnarray}\label{eq:5.3}
{\cal H} ( t, x, u, p, q, r ) := \left ( b ( t, x, u ), p \right )_H
+\left( g ( t, x,  u), q \right)_H
+ \left( \sigma ( t,  x,  u), r(t) \right)_{l^2(H)}
+ l ( t, x, u ) .
\end{eqnarray}
Using Hamiltonian ${\cal H}$,
the adjoint equation \eqref{eq:4.4}
can be written in the following form:

\begin{eqnarray}\label{eq:5.4}
\begin{split}
   \left\{\begin{array}{ll}
d\bar p(t)=&-\bigg[A^*(t)\bar p(t)+B(t)^*\bar q(t)+\bar {\cal H}_{x} (t)\bigg]dt+\bar q(t)dW(t)+\displaystyle
\su\bar r^i(t)dH^i(t),
~~~~0\leqslant t\leqslant T,
\\ \bar p(T)=&\Phi_x( \bar X(T)),
  \end{array}
 \right.
 \end{split}
  \end{eqnarray}
where we denote
\begin{eqnarray}\label{eq:5.6}
\bar {\cal H} (t) \triangleq {\cal H} ( t,
\bar x (t), \bar u (t), \bar p (t),
\bar q (t),
\bar r(t) ).
\end{eqnarray}

\section{  Stochastic Maximum Principle}
\subsection {Variation of the State Process and Cost Functional}
Let$( {\bar u} (\cdot); {\bar X} (\cdot) )$ be an optimal pair of Problem
\ref{pro:2.1}.  Define  a
convex perturbation  of $\bar u(\cdot)$ as follows:
\begin{eqnarray*}
u^\epsilon ( \cdot ) \triangleq {\bar u} ( \cdot ) + \epsilon ( v ( \cdot ) - {\bar u} ( \cdot) ) ,
\quad 0 \leq \epsilon \leq 1,
\end{eqnarray*}
where $v(\cdot)$ is an arbitrarily
admissible control. Since the control domain $\mathscr U$ is
convex,  $u^\eps(\cdot)$ is also an element of
$\cal A.$
 We denote by $ X^\eps(\cdot)$  the state process
corresponding to the control $u^\eps(\cdot).$
Now we introduce the following  first order
variation equation:
\begin{eqnarray}\label{eq:5.8}
\left\{
\begin{aligned}
d Y (t) =& \ [ A (t) Y(t) + b_x ( t, \bar X(t),
\bar u(t) )Y(t) + b_u ( t,\bar X(t), \bar u(t) )(v(t)-\bar u(t)) ] d t
\\&+[B(t) Y(t)
+  g_x ( t,\bar X(t), \bar u(t))Y(t) + g_u ( t,
\bar X(t), \bar u(t) )(v(t)-\bar u(t))] d W (t)
\\& +\su\bigg[\sigma^i_x ( t, \bar X(t-), \bar u(t) )Y(t-)
 +\sigma^i_u ( t, \bar X(t-), \bar u(t) )(v(t)-\bar u(t))\bigg]
  dH^i(t), \\
Y(0) =& 0.
\end{aligned}
\right.
\end{eqnarray}
Under Assumption \ref{ass:2.5}, by
Theorem \ref{thm:3.1}, we see that
the variation equation
\eqref{eq:5.8} has a unique
solution  $Y(\cdot)\in M_{\mathscr{F}}^2(0,T;V).$

\begin{lem}\label{lem:6.1}
    Let  Assumption \ref{ass:2.5} be satisfied. Then we have the following estimates:
   \begin{eqnarray}\label{eq:6.2}
 \sup_{0 \leq t \leq T}{\mathbb E} \bigg [ \| X^\epsilon (t) - {\bar X} (t) \|^2_H \bigg ]+{\mathbb E} \bigg [ \int_0^T \| X^\epsilon (t) - {\bar X} (t) \|^2_V d t \bigg ] = O (\epsilon^2),
\end{eqnarray}

  \begin{eqnarray}\label{eq:5.5}
  \sup_{0 \leq t \leq T} {\mathbb E} \bigg [ \| X^\epsilon (t) - {\bar X} (t)-\eps Y(t) \|^2_H \bigg ]+
{\mathbb E} \bigg [ \int_0^T \| X^\epsilon (t) - {\bar X} (t)-\eps Y(t) \|^2_V d t \bigg ] = o (\epsilon^2) \ .
\end{eqnarray}
 \end{lem}

 \begin{proof}
   From the estimate \eqref{eq:4.5},
   we have

   \begin{eqnarray}
   \begin{split}
   \sup_{0 \leq t \leq T} {\mathbb E} \bigg [ \| X^\epsilon (t) - {\bar X} (t)) \|^2_H \bigg ]+
{\mathbb E}
\bigg [ \int_0^T \| X^\eps (t) - { \bar X} (t) \|_V^2 d t \bigg ]
 &\leq K {\mathbb E} \bigg [ \int_0^T \| u^\eps(t) -\bar u(t) \|_U^2 d t \bigg ]
 \\&=K\eps^2 {\mathbb E} \bigg [ \int_0^T \| v(t) -\bar u(t) \|_U^2 d t \bigg]
 \\&=O(\eps^2).
 \end{split}.
\end{eqnarray}
 Denote
 \begin{eqnarray}
   \Xi^\eps(t):= X^\eps(t)-\bar X(t)-\eps Y(t).
 \end{eqnarray}
From   Taylor expanding, we have
\begin{eqnarray} \label{eq:6.5}
  \left\{
  \begin{aligned}
   d \Xi^\eps (t) = & \ [ A (t) \Xi^\eps (t)
   +b_x(t, \bar X(t), \bar u(t))\Xi^\eps (t)
   + \alpha^\eps(t)]dt
   + [B(t)\Xi^\eps(t)+g_x(t, \bar X(t), \bar u(t))\Xi^\eps (t)
   + \beta^\eps(t) ]d W(t)\\
   &+  \su \bigg[\sigma^i_x(t,  \bar X(t-), \bar u(t))\Xi^\eps (t)
   + \gamma^{i\eps}(t ) \bigg]dH^i(t),  \\
X (0) = & \  x , \quad t \in [ 0, T ],
  \end{aligned}
  \right.
\end{eqnarray}
 where
 \begin{eqnarray} \label{eq:5.16}
\left\{
\begin{aligned}
\alpha^\eps(t)= &\int_0^1\bigg[\big(b_x( t, \bar X (t)+\lambda(X^\eps(t)-\bar X(t)), \bar u (t)+\lambda(u^\eps(t)-\bar u(t)))-b_x(t, \bar X(t), \bar u(t))\big)(X^\eps(t)-\bar X(t))
   \\&+\big(b_u( t,\bar X (t)+\lambda(X^\eps(t)-\bar X(t)),\bar u (t)+\lambda(u^\eps(t)-\bar u(t)))-b_u(t, \bar X(t), \bar u(t))\big)(u^\eps(t)-\bar u(t))\bigg]d\lambda,
   \\
   \beta^\eps(t)= &\int_0^1\bigg[\big(g_x( t, \bar X (t)+\lambda(X^\eps(t)-\bar X(t)), \bar u (t)+\lambda(u^\eps(t)- \bar u(t)))-g_x(t, \bar X(t), \bar u(t))\big)(X^\eps(t)-\bar X(t))
   \\&+\big(g_u( t,\bar X (t)+\lambda(X^\eps(t)-\bar X(t)),\bar u (t)+\lambda(u^\eps(t)-
   \bar u(t)))-g_u(t, \bar X(t), \bar u(t))\big)(u^\eps(t)-\bar u(t))\bigg]d\lambda,\\
   \gamma^{i\eps}(t )= &\int_0^1\bigg[\big(\sigma^i_x( t,   \bar X (t-)+\lambda(X^\eps(t-)-\bar X(t-)), \bar u (t)+\lambda(u^\eps(t)-\bar u(t))-
   \sigma^i_x(t,  \bar X(t-), \bar u(t))\big)(X^\eps(t-)-\bar X(t-))
   \\&+\big(\sigma^i_u( t, \bar X (t-)+\lambda(X^\eps(t-)-\bar X(t-)),\bar u (t)+\lambda(u^\eps(t)-\bar u(t)))-
   \sigma^i_x(t,  \bar X(t-), \bar u(t))\big)(u^\eps(t)-\bar u(t))\bigg]d\lambda.
\end{aligned}
\right.
\end{eqnarray}
 From the estimates \eqref{eq:3.4},
 \eqref{eq:6.2} and  Lebesgue dominated convergence theorem, we get that

 \begin{eqnarray}
&& \sup_{0 \leq t \leq T} {\mathbb E} \bigg [ \| \Xi (t) \|^2_H \bigg ]+{\mathbb E} \bigg [ \int_0^T \| \Xi(t) \|^2_V d t \bigg ]\nonumber
\leq K\bigg\{ \mathbb E\bigg[ \int_0^T || \alpha^\eps(t)||^2_Hdt  \bigg]
+\mathbb E\bigg[\int_0^T | |\beta^\eps(t)||^2_Hdt\bigg]
 +\mathbb E\bigg[ \int_0^T
 ||\gamma^\eps(t )||^2_{l^2(H)}dt \bigg]
 \bigg\}=o(\eps).
\end{eqnarray}
The proof is complete.
\end{proof}

 \begin{lem} \label{lem:6.2}
    Let  Assumption \ref{ass:2.5} be satisfied.                Let $( {\bar u} (\cdot); {\bar X} (\cdot) )$ be
 an optimal pair of
 Problem \ref{pro:2.1} associated with
  the first order variation process $Y(\cdot)$ (see \eqref{eq:5.8}). Then,

  \begin{eqnarray}\label{eq:6.9}
  \begin{split}
    J(u^\eps(\cdot))-J(\bar u(\cdot))=&
    \eps\mathbb E\bigg[(\Phi_x( \bar X(T)), Y(T))_H\bigg]
      + \eps\mathbb E\bigg[\int_0^T
      ( l_x(t,\bar X(t), \bar u(t) ), Y(t))_Hdt\bigg]
      \\&+ \eps\mathbb E\bigg[\int_0^T
      ( l_u(t, \bar X(t), \bar u(t)), v(t)-u(t))_Udt\bigg]+o(\eps).
  \end{split}
\end{eqnarray}
 \end{lem}

\begin{proof}

From the definition of the cost functional (see \eqref{eq:4.2}), we have
\begin{eqnarray} \label{eq:6.10}
J(u^\eps(\cdot))- J(\bar u(\cdot))
={\mathbb E} \bigg [ \int_0^T
\Big(l ( t, X^\eps (t), u^\eps (t) ) - l ( t, \bar X(t), \bar u (t) \Big) d t\bigg ]
+ \mathbb E\bigg[ \Phi ( X^\eps (T) )
-\Phi ( \bar X (T) ) \bigg ]=:I_1+I_2{\color{blue},}
 \end{eqnarray}
where
\begin{eqnarray}
  \begin{split}
    I_1=&{\mathbb E} \bigg [ \int_0^T
\Big(l ( t, X^\eps (t), u^\eps (t) ) - l ( t, \bar X(t), \bar u (t) \Big) d t\bigg ]
,
  \\  I_2=&\mathbb E\bigg[ \Phi ( X^\eps (T) )
-\Phi ( \bar X (T) ) \bigg ].
  \end{split}
\end{eqnarray}
Let us concentrate on $I_1.$
In terms of Taylor expanding, Lemma \ref{lem:6.1} and
the control convergence theorem, we have
\begin{eqnarray} \label{eq:6.11}
  \begin{split}
    I_1= &\mathbb E\bigg[\int_0^T\int_0^1\big(l_x( t, \bar X (t)+\lambda(X^\eps(t)-\bar X(t)), \bar u (t)+\lambda(u^\eps(t)-\bar u(t)))-l_x(t, \bar X(t), \bar u(t))\big)(X^\eps(t)-\bar X(t))d\lambda dt\bigg]
   \\&+ E\bigg[\int_0^T\int_0^1\big(l_u( t,\bar X (t)+\lambda(X^\eps(t)-\bar X(t)),\bar u (t)+\lambda(u^\eps(t)-\bar u(t)))-l_u(t, \bar X(t), \bar u(t))\big)(u^\eps(t)-\bar u(t)) d\lambda dt\bigg]
   \\&+\mathbb E\bigg[\int_0^Tl_x(t, \bar X(t), \bar u(t))\big)\Xi^\eps(t)dt\bigg]
   +\eps \mathbb E\bigg[\int_0^Tl_x(t, \bar X(t), \bar u(t))\big)Y(t) dt\bigg]
   \\&+ \eps E\bigg[\int_0^T l_u(t, \bar X(t), \bar u(t))\big)(u(t)-\bar u(t)) dt\bigg]
  \\ =& \eps \mathbb E\bigg[\int_0^Tl_x(t, \bar X(t), \bar u(t))Y(t) dt\bigg]+ \eps E\bigg[\int_0^T l_u(t, \bar X(t), \bar u(t))(u(t)-\bar u(t)) dt\bigg]
   \\&+o(\eps).
  \end{split}
\end{eqnarray}

Similarly,  we have

\begin{eqnarray} \label{eq:6.12}
  \begin{split}
    I_2= \eps \mathbb E\bigg[\Phi _x(\bar X(T))Y(T)\bigg]+o(\eps).
  \end{split}
\end{eqnarray}

Then putting \eqref{eq:6.11} and
\eqref{eq:6.12} into \eqref{eq:6.10},
we get \eqref{eq:6.9}. The proof is complete.

\end{proof}

\subsection{Main Results}
Now we are in position to  state and prove the maximum principle for  Problem {\ref{pro:2.1}}.
\begin{thm}[{\bf Maximum Principle}]
\label{thm:4.3}
Let  Assumption \ref{ass:2.5} be satisfied.
Let $( {\bar u} (\cdot); {\bar X} (\cdot) )$ be an optimal pair of
Problem \ref{pro:2.1} associated with
the adjoint processes $( {\bar  p} (\cdot), {
\bar q} (\cdot), \bar r(\cdot) ).$
Then the following minimum condition holds:
\begin{eqnarray}\label{eq:4.9}
\big( {\cal H}_u (t,\bar X(t-), \bar u(t),\bar p(t-), \bar q(t),
\bar r(t)), v - {\bar u} (t) \big)_U \geq 0 ,\quad \quad\forall v \in {\mathscr U}, for~ a.e.~t \in [ 0, T ], {\mathbb P}-a.s.
\end{eqnarray}
\end{thm}
\begin{proof}
Recalling the adjoint equation
\eqref{eq:5.4}
 and the first order variational
 equation \eqref{eq:5.8}, and then applying It\^o formula to
     $(\bar p(t), Y(t))_H,$  we have
  \begin{eqnarray} \label{eq:6.14}
    \begin{split}
        &\mathbb E[(\Phi_x( \bar X(T)),
         Y(T))_H ]
      + \mathbb E\bigg[\int_0^T
      ( l_x(t,\bar X(t), \bar u(t) ), Y(t))_Hdt\bigg]
      \\=& \mathbb E\bigg[\int_0^T
      \bigg(v(t)-\bar u(t),  b_u^*(t,\bar X(t), \bar u(t))\bar p(t)
      + g_u^*(t,\bar X(t), \bar u(t))\bar q(t)
      +
      \su\sigma{^i}_u^*(t, \bar X(t), \bar u(t))\bar r^i(t ) \bigg)_Udt\bigg].
    \end{split}
  \end{eqnarray}

Since $\bar u(\cdot) $ is the optimal
control,
from \eqref{eq:6.9}, the duality relation
\eqref{eq:6.14} and the definition of the
Hamiltonian ${\cal H}$ (see \eqref{eq:5.3})
, we have
\begin{eqnarray}\label{eq:5.6}
  \begin{split}
  0\leq&   \lim_{\eps\longrightarrow 0^{+}}\frac {J(u^\eps(\cdot))-J(\bar u(\cdot))}{\eps}
 \\ =&\mathbb E[\big(\Phi_x( \bar X(T)),
         Y(T)\big)_H ]
      + \mathbb E\bigg[\int_0^T
      ( l_x(t,\bar X(t), \bar u(t) ), Y(t))_Hdt\bigg]+ \mathbb E\bigg[\int_0^T
      ( l_u(t, \bar X(t), \bar u(t)), v(t)-u(t))_Udt\bigg]
      \\=&\mathbb E\bigg[\int_0^T
      \bigg(v(t)-\bar u(t),  b_u^*(t,\bar X(t), \bar u(t))\bar p(t)
      + g_u^*(t,\bar X(t), \bar u(t))\bar q(t)
      +
      \su\sigma{^i}_u^*(t, \bar X(t), \bar u(t))\bar r^i(t ) \bigg)_Udt\bigg]
      \\&+\mathbb E\bigg[\int_0^T
      ( l_u(t, \bar X(t), \bar u(t)), v(t)-\bar u(t))_Udt\bigg]
      \\=&E\bigg[\int_0^T
      \bigg(v(t)-\bar u(t), {\cal H}_u (t,\bar X(t), \bar u(t),\bar p(t), \bar q(t),
\bar r(t)) )_Udt\bigg].
  \end{split}
\end{eqnarray}
This  implies the minimum condition
\eqref{eq:4.9} holds  since $v(\cdot)$
is any given admissible control.
  \end{proof}

\section{Verification  Theorem }

In the following, we give  a sufficient condition of optimality for the existence of an optimal control of
Problem \ref{pro:2.1},
which is the so-called verification theorem.

\begin{thm}[{\bf Verification  Theorem}]\label{thm:4.4}
Let  Assumption \ref{ass:2.5} be satisfied.
Let $( {\bar u} (\cdot); {\bar X} (\cdot) )$ be an admissible pair of
Problem \ref{pro:2.1} associated with
the adjoint processes $( \bar p(\cdot),
\bar q (\cdot ), \bar r(\cdot)).$ Suppose that
${\cal H} ( t, x, u, { \bar p} (t), {\bar q} (t), \bar r(t) )$ is convex in $( x, u)$,
and $\Phi (x)$ is convex in $x$, moreover assume
that  the following  optimality condition holds for almost all $( t, \omega ) \in [ 0, T ] \times \Omega$:
\begin{eqnarray} \label{eq:4.10}
 {\cal H} ( t, {\bar X} (t),
\bar u(t), { \bar p} (t), {\bar q} (t), \bar r(t)) = \min_{ u \in {\mathscr U} } {\cal H} ( t,  {\bar X} (t),
u, { \bar p} (t), {\bar q} (t), \bar r(t)).
\end{eqnarray}
 Then $( {\bar u} (\cdot); {\bar X} (\cdot) )$
is an optimal pair of Problem \ref{pro:2.1}.
\end{thm}
\begin{proof}
Let  $( {u} (\cdot); {X} (\cdot) )$  be an any   given admissible pair. To simplify our notations, we define
\begin{eqnarray} \label{eq:6.1}
\begin{split}
& b (t) \triangleq b ( t, X (t), u (t)), {\bar b} (t) \triangleq b ( t, {\bar X} (t), {\bar u} (t)), \\
& g (t) \triangleq g ( t, X (t), u (t)), {\bar g} (t) \triangleq g ( t, {\bar X} (t) ,{\bar u} (t)),
\\
& \sigma{^i} (t ) \triangleq \sigma{^i} ( t,  X (t-), u (t)), {\bar \sigma{^i}} (t) \triangleq \sigma{^i} ( t,
 {\bar X} (t-) ,{\bar u} (t)),
\\
& {\cal H} (t) \triangleq {\cal H} ( t, X (t), u (t), {\bar p} (t), {\bar q} (t), \bar r(t) ),
\\
& { \bar{\cal H}} (t) \triangleq {\cal H} ( t,
\bar X (t), \bar u (t), {\bar p} (t), {\bar q} (t), \bar r(t) ) .
\end{split}
\end{eqnarray}
From the definitions of the cost functional
  $J(u(\cdot))$ and the Hamiltonian  ${\cal H}$ (see \eqref{eq:4.2} and \eqref{eq:5.3}),
  we can represent $J(u(\cdot))-J(\bar u(\cdot))$ as follows:
\begin{eqnarray}\label{eq:7.3}
J ( u (\cdot) ) - J ( {\bar u} (\cdot) )
&=& {\mathbb E} \bigg [ \int_0^T \bigg ( {\cal H} (t) - {\bar {\cal H}} (t)
- ( {\bar p} (t), b(t) - {\bar b} (t) )_H \nonumber - ( {\bar q} (t), g (t) - {\bar g} (t) )_H
\\&&- \su( {\bar r}^i (t ), \sigma{^i} (t ) - {\bar \sigma{^i}} (t ) )_H    \bigg ) d t \bigg ]+ {\mathbb E} \bigg[ \Phi ( X (T))
- \Phi ( {\bar X} (T) )\bigg] .
\end{eqnarray}

Then recalling the
 adjoint equation \eqref{eq:4.4}  and applying  It\^o's formula to $( {\bar p} (t), X (t) - {\bar X} (t) )_H$, we  get that
\begin{eqnarray}\label{eq:4.12}
&&{\mathbb E}\bigg [ \int_0^T  \bigg (( { \bar p} (t), b (t) - {\bar b} (t) )_H
+ ( {\bar q} (t), g (t) - {\bar g} (t) )_H
+ \su( { \bar r}^i (t ), \sigma{^i} (t ) - {\bar \sigma{^i}} (t ) )_H
\bigg ) d t \bigg ] \nonumber \\
~~~~~~~~~&=& {\mathbb E} \bigg [ \int_0^T ( {\bar {\cal H}}_x (t) , X (t) - {\bar X} (t) )_H d t \bigg ]+ {\mathbb E} \Big[ (\Phi_{x} ( \bar X  (T)), X (T) - {\bar X} (T) )_H \Big] .
\end{eqnarray}
Then substituting \eqref{eq:4.12} into \eqref{eq:7.3} leads to

\begin{eqnarray}\label{eq:6.6}
J ( u (\cdot) ) - J ( {\bar u} (\cdot) )
&=& {\mathbb E} \bigg [ \int_0^T \bigg
( {\cal H} (t) - {\bar {\cal H}} (t)
- ( {\bar {\cal H}}_x (t),
 X(t) - {\bar X} (t) )_H \bigg )d t \bigg ] \nonumber \\
&& + {\mathbb E} [ \Phi ( X (T) ) - \Phi ( {\bar X} (T) )
- ( \Phi_x ( {\bar X} (T) ), X(T) - {\bar x} (T) )_H ] .
\end{eqnarray}
On the other hand,   the convexity of
${\cal H}(t)$ and $\Phi (x)$ yields
\begin{eqnarray}\label{eq:7.7}
{\cal H} (t) - {\bar {\cal H}} (t) &\geq& ( {\bar {\cal H}}_x (t), X (t) - {\bar X} (t) )_H
+( {\bar {\cal H}}_{u} (t), u ( t ) - {\bar u} ( t) )_U ,
\end{eqnarray}
and
\begin{eqnarray}\label{eq:7.8}
\Phi ( X (T) ) - \Phi ( {\bar X} (T) ) \geq ( \Phi_x ( {\bar X} (T) ), x (T) - {\bar x} (T) )_H .
\end{eqnarray}
In addition,  the optimality condition \eqref{eq:4.10}  and the convex optimization principle
(see Proposition 2.21 of \cite{ET1976} ) yield  that  for almost all $( t, \omega ) \in
[ 0, T ] \times \Omega$,
\begin{eqnarray}\label{eq:5.5}
( {\bar {\cal H}}_u (t), u (t) - {\bar u} (t) )_U\geq 0 .
\end{eqnarray}
Then putting \eqref{eq:7.7}, \eqref{eq:7.8} and \eqref{eq:5.5} into \eqref{eq:6.6}, we get that
\begin{eqnarray}
 J ( u (\cdot) ) - J ( {\bar u} (\cdot) )\geq 0.
\end{eqnarray}
Therefore,  since $u(\cdot)$ is arbitrary,  ${\bar u} (\cdot)$ is an optimal control process and $( {\bar u} (\cdot); {\bar X} (\cdot) )$ is an optimal pair.  The proof is
complete.
\end{proof}
\section{Application}
In this section, we will apply our theoretical results to solve a specific
exampl  i.e., an
optimal control problem for a controlled  Cauchy problem
driven by Teugels martingales.

First of all,
let us recall some preliminaries of Sobolev spaces.
For $m = 0, 1$, we define the space $H^m \triangleq \{ \phi:
\partial_z^\alpha \phi \in L^2 ( {\mathbb R}^d ), \ \mbox {for any}
\ \alpha: =( \alpha_1, \cdots, \alpha_d ) \ \mbox {with} \ |\alpha|
:= | \alpha_1 | + \cdots + | \alpha_d | \leq m \}$ with the norm
\begin{eqnarray*}
\| \phi \|_m \triangleq \left \{ \sum_{ |\alpha| \leq m } \int_{{\mathbb R}^d}
| \partial_z^\alpha \phi (z) |^2 d z \right \}^{\frac{1}{2}} .
\end{eqnarray*}
We denote by $H^{-1}$ the dual space of $H^1$. We set $V = H^1$, $H
= H^0$, $V^* = H^{-1}$. Then $( V, H, V^* )$ is a Gelfand triple.

  We choose  control domain ${\mathscr U} = U = H$. The admissible control
set $\cal A$  is defined as  $ { M}^2_{\mathscr F} ( 0, T; U ).$
 For any admissible control $u(\cdot, \cdot)
 \in { M}^2_{\mathscr F} ( 0, T; U ),$
  we consider a controlled  Cauchy problem, where the system
is given by a stochastic partial differential equation driven by
Brownian motion $W$
and Poisson random
martingale $ $ in  the following
divergence form:
\begin{eqnarray}\label{eq:6.13}
\left\{
\begin{aligned}
d y ( t,z ) = & \ \big \{ \partial_{z^i} [ a^{ij} ( t, z )
\partial_{z^j} y ( t, z ) ]
+ b^i ( t, z ) \partial_{z^i} y ( t, z )+ c ( t, z ) y ( t, z ) + u ( t, z ) \big \} d t \\
& + \{\partial_{z^i} [\eta^i( t, z )y ( t, z )]  +\rho ( t, z ) y ( t, z )+
u ( t,z) \} d W (t)
+  \su[  \Gamma^i ( t,  z ) y ( t, z )+
u ( t, z ) ] dH^i(t) ,\\
y ( 0, z ) = & \ \xi(z)\in \mathbb R^d{\color{blue},}\quad ( t, z ) \in [ 0, T ] \times {\mathbb R}^d,
\end{aligned}
\right.
\end{eqnarray}
 where the coefficients $a^{ij}, b^i, \eta^i, c,\rho:
[ 0, T ]\times \Omega\times \mathbb R^d \rightarrow \mathbb R$ and $\Gamma^i:
[ 0, T ]\times \Omega \times \mathbb R^d $  are given  random mappings
and satisfy the suitable measurability.
Here we also use the Einstein summation convention to
$\partial_{z^i} [ a^{ij} ( t, z )
\partial_{z^j} y ( t, z ) ], b^i ( t, z ) \partial_{z^i} y ( t, z )$
and $\partial_{z^i} [\eta^i( t, z )y ( t, z )].$

For any admissible  control $u(\cdot, \cdot)\in \cal A$, the following
definition gives the generalized weak solution to  \eqref{eq:6.13}.

\begin{defn}
An $\mathbb R$-valued, ${\mathscr P} \times {\mathscr B} (\mathbb R^d
)$-measurable process $y (\cdot, \cdot)$ is called a solution to
\eqref{eq:6.13}, if $y (\cdot, \cdot) \in {\cal M}_{\cal F}^2 ( 0, T;
H)$  such that for every $\phi \in H$ and a.e. $( t,
\omega) \in [ 0, T ] \times \Omega $, it holds that
\begin{eqnarray}
\begin{split}
\int_{\mathbb R^d} y (t,z) \phi(z)dz  =& \ \int_{\mathbb R^d}\xi (  z )\phi(z)dz -
\int_0^t \int_{\mathbb R^d} a^{ij} ( s, z )
\partial_{z^j} y ( s, z ) \partial_{z^i}\phi(z) dz d s
+\int_0^t \int_{\mathbb R^d} \bigg [ b^i ( s, z ) \partial_{z^i} y ( s, z ) \\
& + c ( s, z ) y ( s, z ) + u(s, z) \bigg ] \phi(z) dz d s -\int_0^t \int_{\mathbb R^d}  \eta^i ( s, z ) y ( s, z )\partial_{i}\phi(z)dzd W (s)
\\&+ \int_0^t \int_{\mathbb R^d}
\bigg[ \rho ( s, z ) y ( s, z )+
u ( s, z )\bigg] \phi(z) dzd W (s) + \su\int_0^t\int_{\mathbb R^d} 
\bigg[  \Gamma^i ( s,  z ) y ( s, z )+
u ( s, z ) \bigg]\phi(z)dzdH^i(t).
\end{split}
\end{eqnarray}
\end{defn}

For any admissible control process $u (\cdot, \cdot)$ and the solution $y (\cdot, \cdot)$ of the
corresponding state equation \eqref{eq:6.13}, the objective of the control problem is to minimize the following cost functional
\begin{eqnarray}\label{eq:8.2}
 J ( u (\cdot) )={\mathbb E}\bigg [ \int_{{\mathbb R}^d} y^2 ( T, z ) d z + \iint_{[
0, T ] \times {{\mathbb R}^d}} y^2 ( s, z ) d s  d z + \iint_{[ 0, T
] \times {{\mathbb R}^d}} u^2 ( s, z ) d s d z \bigg ] .
\end{eqnarray}

To make the control problem well-defined,
we make the following assumptions
on the coefficients $a$, $b$, $c$, $\eta$, $\rho$, $\Gamma$, for some fixed constants $K \in ( 1,
\infty )$ and $\kappa \in ( 0,1 )$:

\begin{ass}\label{ass:6.3}
The functions $a$, $b$, $c$, $\eta$, and $\rho$  are
${\mathscr P} \times {\mathscr B} ({\mathbb R}^d )$-measurable with
values in the set of real symmetric
$d \times d$ matrices,  ${\mathbb
R}^{d}$, ${\mathbb R}$, ${\mathbb R}^d$
and  ${\mathbb R}$, respectively, and are bounded by $K$.
The function $\Gamma$ is
${\mathscr P}
\times {\mathscr B} (E)\times {\mathscr B} ({\mathbb R}^d )$-measurable with
value ${l^2(\mathbb
R)}$ and is bounded by $K$. $\xi\in L^2 ( {\mathbb R}^d).$
\end{ass}

\begin{ass}\label{ass:6.4}
The super-parabolic condition holds, i.e.,
\begin{eqnarray*}
\kappa I+\eta( t, z )(\eta( t, z ))^*\leq 2 a ( t, \omega, z ) \leq K I , \quad \forall ( t, \omega, z ) \in
[ 0, T ] \times \Omega \times {\mathbb R}^d ,
\end{eqnarray*}
where $I$ is the $(d \times d)$-identity matrix.
\end{ass}
In order to  apply our abstract theoretical results in Section 6 and 7 to our optimal control problem, now we  begin to transform
\eqref{eq:6.13} into a SEE  driven by 
Teugels martingales in
the form of \eqref{eq:3.1}.
Set
\begin{eqnarray*}
&&X(t)\triangleq y(t, \cdot),
\\&& (A (t) \phi) (z) \triangleq  \partial_{z^i} [ a^{ij} ( t, z )
\partial_{z^j} \phi (z) ]
+b^i ( t, z ) \partial_{z^i} \phi (z)+c ( t, z ) \phi (z)  , \quad \forall \phi \in V , \\
&&( B (t) \phi) (z) \triangleq \partial_{z^i}[\eta^i (t,z)\phi (z)]+\rho ( t, z ) \phi ( z ) ,\quad \forall \phi \in V,
\\&& b(t,\phi, u)\triangleq u, \quad \forall \phi \in H, u\in {\cal U},
\\&& g(t,\phi, u)\triangleq u, \quad \forall \phi \in H, u\in {\cal U},
\\&& \sigma{^i}(t, \phi, u)\triangleq\Gamma^i ( t ) \phi+u,  \quad\forall \phi \in H, u\in \cal U,
 \\&& l(t,\phi, u)\triangleq  (\phi, \phi)_H+
 (u, u)_{U},  \quad\forall \phi \in H, u\in \cal U,
 \\&& \Phi(\phi)\triangleq (\phi, \phi)_H, \quad\forall \phi \in H.
\end{eqnarray*}

In the Gelfand triple $( V, H, V^* )$, using the above notations, we can rewrite the state equation \eqref{eq:6.13} as follows:
\begin{eqnarray} \label{eq:8.3}
  \left\{
  \begin{aligned}
   d X (t) = & \ [ A (t) X (t) + b ( t, X (t), u(t)) ] d t
+ [B(t)X(t)+g( t, X (t), u(t)) ]d W(t)
 \\&\quad + \su\sigma{^i} (t, X(t),u(t))dH^i(t),  \\
X (0) = & \  x , \quad t \in [ 0, T ],
  \end{aligned}
  \right.
\end{eqnarray}
and the cost functional \eqref{eq:8.2}
can be  rewritten as
\begin{eqnarray}\label{eq:8.4}
J ( u (\cdot) ) = {\mathbb E} \bigg [ \int_0^T l ( t, x (t), u (t) ) d t
+ \Phi ( x (T) ) \bigg ]{\color{blue},}
\end{eqnarray}
where we set

\begin{eqnarray}
\begin{split}
  &l(t,x,u)\triangleq(x, x)_H+(u, u)_H, \forall x\in H, u\in U,
\\&
\Phi(x)\triangleq(x,x)_H, \forall x\in H.
\end{split}
\end{eqnarray}
Thus this optimal control problem is  transformed  into Problem
\ref{pro:2.1} as a special case.
Under Assumptions \ref{ass:6.3}-\ref{ass:6.4},
 it  is easy to check that the coefficients of  this optimal control problem  satisfy
Assumptions \ref{ass:2.5}.  So  in this
cas   Theorem \ref{thm:4.3} and
\ref{thm:4.4} hold.
Moreover, from
the a priori estimate \eqref{eq:3.5},
it is easy to see that the cost functional  $J(u(\cdot))$ is
the strictly convex, coerciv  lower-semi continuous functional defined
on the reflexive Banach space
${ M}^2_{\mathscr F} ( 0, T; U ).$
Therefor  the uniqueness and existence of the optimal control
can be  obtained by the convex optimality
 principle (see Proposition 2.12 of \cite{ET1976}).
 Let $(\bar u(\cdot), \bar X(\cdot))$
 be the optimal pair. In the following, we will give the
 duality characterization of the optimal
 control $\bar u(\cdot)$ by the maximum
 principle. More precisely, in this case the corresponding Hamiltonian ${\cal H}$ becomes
\begin{eqnarray}\label{eq4.2}
{\cal H} ( t, x, u, p, q, r ) := \left ( u, p \right )_H
+\left(u , q \right)_H
+ \left( \Gamma ( t) x+u, r(t ) \right)_{l^2(H)}
+ (x,x)_H+(u,u)_H .
\end{eqnarray}
 Let $( \bar u (\cdot); \bar X (\cdot) )$
 be an  optimal pair. Then
 the corresponding adjoint equation
becomes
\begin{eqnarray}\label{eq:8.6}
\begin{split}
   \left\{\begin{array}{ll}
d\bar p(t)=&-\bigg[A^*(t)\bar p(t)
+B^*(t)\bar q(t)+\displaystyle
\su
\Gamma^{i*}(t )\bar r^i(t)
+ 2 \bar X(t)\bigg]dt
+\bar q(t)dW(t)+\displaystyle \su\bar r^i(t)dH^i(t),0\leqslant t\leqslant T,
\\ \bar p(T)=&2\bar X(T),
  \end{array}
 \right.
 \end{split}
  \end{eqnarray}
where
\begin{eqnarray*}
 &&A^*(t) \phi (z) \triangleq - \partial_{z^i} [ a^{ij}
( t, z ) \partial_{z^j} \phi (z) ] + \partial_{z^i}[b^i ( t, z )
\phi (z)] + c ( t, z ) \phi (z) , \quad \forall \phi \in V,\\
&&B^*(t) \phi (z) \triangleq   -\eta^i (t,z)\partial_{z^i}\phi (z),  \quad \forall  \phi \in H,\\
&&\Gamma^{i*}(t ) \phi (z) \triangleq   \Gamma^i(t, z)\phi (z),  \quad \forall  \phi \in H.
\end{eqnarray*}
Since ${\mathscr U} = U$, there is no constraint on the control and  therefore  the minimum condition \eqref{eq:4.9}
becomes
\begin{eqnarray}
{\cal H}_u (t,\bar X(t-), \bar u(t),
\bar p(t-), \bar q(t),
\bar r(t))= 0,
\end{eqnarray}
which imply that
\begin{eqnarray}\label{eq6.4}
&& 2 {\bar u} (t) + \bar  { p} (t-) + \bar q (t)
+  \bar r(t ) = 0,
\end{eqnarray}a.e. $t \in [ 0, T ]$, ${\mathbb P}$-a.s..
 Thus  the optimal control ${\bar u}
(\cdot)$ is given by
\begin{eqnarray*}
{\bar u} (t) = -\frac{1}{2} \bigg [
\bar { p} (t-) + \bar { q} (t)+  \bar r(t ) \bigg] .
\end{eqnarray*}

\begin{rmk}
  The above example
  can be regarded as
  a special case of  the  infinite-dimensional
linear-quadratic control problem 
driven by Teugels  martiangles
which can also be  applied to
some more practical problems
such as the partial observation optimal
control driven by  Teugels martingales and the optimal harvesting problem 
 associated with L\'{e}vy processes and so on. And we will  give
 detailed investigations on these applications in our future
   publication.

\end{rmk}

\section{Conclusion}

In this paper, we have developed an infinite-dimensional optimal control problem of the
stochastic  evolution system 
driven by Teugels martingales.
We have considered the control variable enters the diffusion of the state equation and  the control domain is  convex.
We first provided the existence  uniqueness and continuous dependence theorems of solutions to SEE driven by
Teugels martingales.
Then we established necessary and sufficient conditions for optimal controls in the form of maximum principles by convex variational technique.
As an application, we considered an optimal control problem of a Cauchy problem for a controlled stochastic partial differential equation and obtained the dual characterization of the optimal control in terms of
the solution to the corresponding stochastic Hamiltonian system.
And further investigates will be carried out on the optimal control problem under convex control domain assumption and
more  practical applications in our future
publications  

\bibliographystyle{amsplain}

\vspace{1mm}




\begin{thebibliography}{10}

\bibitem {Al}Albeverio, S., Wu, J. L., \& Zhang, T. S. (1998). Parabolic SPDEs driven by Poisson white noise.
    {\it Stochastic Processes and their Applications,} 74(1), 21-36.



\bibitem{BEE03}
Bahlali, K., Eddahbi, M., \& Essaky, E. (2003). BSDE associated with L¨¦vy processes and application to PDIE. {\it International Journal of Stochastic Analysis,} 16(1), 1-17.



\bibitem{Ben830}
Bensoussan, A. (1983).
\newblock Stochastic maximum principle for distributed parameter systems.
\newblock {\em Journal of the Franklin Institute}, 315, 387--406.











\bibitem{chenshaokuan}
Chen S, \& Tang S.  (2010). Semi-linear backward stochastic integral partial differential equations driven by a Brownian motion and a Poisson point process.
 {\it arXiv preprint arXiv:1007.3201.}

\bibitem{Chow07}
Chow, P. L. (2014). {\it Stochastic partial differential equations.} CRC Press.




 \bibitem{Da} Da Prato, G., \& Zabczyk, J. (2014). {\it Stochastic equations in infinite dimensions.} Cambridge university press.




 \bibitem{Du}Du, K., \& Meng, Q. (2013). A maximum principle for optimal control of stochastic evolution equations.
     {\it SIAM Journal on Control and Optimization,} 51(6), 4343-4362.




\bibitem{ET1976} Ekeland, I., T\'{e}mam, R. (1976). Convex Analysis and Variational Problems. North-Holland, Amsterdam.


\bibitem{Fu} Fuhrman, M., Hu, Y.,
\& Tessitor  G. (2013). Stochastic maximum principle for optimal control of SPDEs.
{\it Applied Mathematics \& Optimization,} 68(2), 181-217.
















 \bibitem{Gy} Gy\"{o}ngy, I., \& Krylov, N. V. (1982). On stochastics equations with respect to semimartingales ii. It\^{o} formula in banach spaces. {\it Stochastics,} 6(3-4), 153-173.

 \bibitem{Hu}Hu, Y. (1991). N-person differential games governed by semilinear stochastic evolution systems. {\it Applied Mathematics and Optimization,} 24(1), 257-271.




 \bibitem{Hupeng}Hu, Y., \& Peng, S. (1990). Maximum principle for semilinear stochastic evolution control systems.
{\it Stochastics and Stochastic Reports,} 33(3-4), 159-180.









 \bibitem{Lu}L\"{u}, Q., \& Zhang, X. (2014). General Pontryagin-type stochastic maximum principle and backward stochastic evolution equations in infinite dimensions. Springer.




 
 

\bibitem{MeTa08}Meng, Q., 
\& Tang, M. (2009). Necessary and sufficient conditions for optimal control of stochastic systems associated with L¨¦vy processes. 
{\it Science in China Series F: Information Sciences,} 52(11), 1982.



\bibitem{MiTa08}
Mitsui, K. I., \& Tabata, Y. (2008). A stochastic linear¨Cquadratic problem with L¨¦vy processes and its application to finance. 
{\it Stochastic Processes and their Applications,} 118(1), 120-152.


\bibitem{NuSc}
Nualart, D., \& Schoutens, W. (2000). Chaotic and predictable representations for L¨¦vy processes.
 {\it Stochastic processes and their applications,} 90(1), 109-122.

\bibitem{NuSc01}
Nualart, D., \& Schoutens, W. (2001). Backward stochastic differential equations and Feynman-Kac formula for L¨¦vy processes, with applications in finance. {\it Bernoulli,} 7(5), 761-776.

\bibitem{Otm06}
El Otmani, M. (2006). Generalized BSDE driven by a L¨¦vy process. International {\it Journal of Stochastic Analysis,} 2006.

\bibitem{Otm08}
El Otmani, M. (2008). Backward stochastic differential equations associated with L\'{e}vy processes and partial integro-differential equations. 
{\it Commun. Stoch. Anal,} 2(2), 277-288.



\bibitem{Ok}{\O}ksendal, B., Prosk  F., \& Zhang, T. (2005). Backward stochastic partial differential equations with jumps and application to optimal control of random jump fields. {\it Stochastics An International Journal of Probability and Stochastic Processes}, 77(5), 381-399.


 \bibitem{Peng}Peng, S. (1990). A general stochastic maximum principle for optimal control problems.
     {\it SIAM Journal on control and optimization,} 28(4), 966-979.








 \bibitem{Pr}Pr\'{e}v\^{o}t, C., \& R\"{o}ckner, M. (2007). {\it A concise course on stochastic partial differential equations (Vol. 1905).} Berlin: Springer.




\bibitem{Ren} Ren Y, Dai H, Sakthivel R.(2013). Approximate controllability of stochastic differential systems driven by a L\'{e}vy process. {\it International Journal of Control, 86(6), 1158-1164.}

\bibitem{ReOt10}
  Ren, Y., \& El Otmani, M. (2010). Generalized reflected BSDEs driven by a L¨¦vy process and an obstacle problem for PDIEs with a nonlinear Neumann boundary condition. {\it Journal of Computational and Applied Mathematics,} 233(8), 2027-2043.

\bibitem{ReFa09}

Ren, Y., \& Fan, X. (2009). Reflected backward stochastic differential equations driven by a L¨¦vy process.
 {\it The ANZIAM Journal,} 50(4), 486-500.
 \bibitem{Ro}R\"{o}ckner, M., \& Zhang, T. (2007). Stochastic evolution equations of jump type: existenc  uniqueness and large deviation principles. {\it Potential Analysis,} 26(3), 255-279.

\bibitem{RSS}Ren, Y., \& Sakthivel, R. (2012). Existenc  uniqueness, and stability of mild solutions for second-order neutral stochastic evolution equations with infinite delay and Poisson jumps.
    {\it Journal of Mathematical Physics,} 53(7), 073517.


\bibitem{Sa}Sakthivel, R., \& Ren, Y. (2012). Exponential stability of second-order stochastic evolution equations with Poisson jumps.
    {\it Communications in Nonlinear Science and Numerical Simulation, 17(12), 4517-4523.}


\bibitem{TaWu09}
Tang, H., \& Wu, Z. (2009). Stochastic differential equations and stochastic linear quadratic optimal control problem {\it with L¨¦vy processes. Journal of Systems Science and Complexity,} 22(1), 122-136.

  
  \bibitem{Tangmeng2016}
  Tang, M.,\& Meng, Q.(2017)
  Stochastic Evolution Equations of Jump Type with Random Coefficients: Existence, Uniqueness and Optimal Control. {\it SCIENCE CHINA Information Sciences    10.1007/s11432-016-9107-1}
    
    \bibitem{TaZh10}
Tang, M., \& Zhang, Q. (2012). Optimal variational principle for backward stochastic control systems associated with L¨¦vy processes.
 {\it Science China Mathematics,} 55(4), 745-761.







 \bibitem{Ya}Yang, X., Zhai, J.,\& Zhang, T. (2015). Large deviations for SPDEs of jump type. {\it Stochastics and Dynamics,} 15(04), 1550026.




    \bibitem{Wal}Walsh, J. B. (2005). Finite element methods for parabolic stochastic PDE¡¯s. {\it Potential Analysis}, 23(1), 1-43.



\bibitem{woy}
Woyczy \'{n}ski W A. (2001).  L\'{e}vy processes in the physical sciences.{\it In L\'{e}vy  processes (pp. 241-266).} Birkh \"{a} user Boston.








 \bibitem{Zhao}Zhao, H., \& Xu, S. (2016). Freidlin-Wentzell 's Large Deviations for Stochastic Evolution Equations with Poisson Jumps. {\it Advances in Pure Mathematics,} 6(10), 676.

 \bibitem{Zhai}Zhai, J., \& Zhang, T. (2015). Large deviations for 2-D stochastic Navier -tokes equations driven by multiplicative L\'{e}vy noises. {\it Bernoulli,} 21(4), 2351-2392.




 \bibitem{Zhou}Zhou, X. (1993). On the necessary conditions of optimal controls for stochastic partial differential equations. {\it SIAM journal on control and optimization,} 31(6), 1462-1478.



































\end{thebibliography}

\end{document}